\documentclass[journal, onecolumn, 11pt]{IEEEtran}
\usepackage[english]{babel}
\usepackage{xspace}
\usepackage{algorithm,algorithmicx}
\usepackage{tcolorbox}
\usepackage{tikz}

\usepackage{framed}
\usepackage{algpseudocode}
\usepackage{amsthm,nccmath}
\usepackage{amsmath,bm}
\usepackage{bigints}
\usepackage{amssymb}
\usepackage{graphicx}
\usepackage{cite}
\usepackage{enumerate}
\usepackage{hyperref}

\usepackage{float}
\usepackage{epstopdf}
\usepackage{soul}
\usepackage{footnote}
\makesavenoteenv{tabular}
\makesavenoteenv{table}
\usepackage{bbm}
\usepackage{amsthm}
\usepackage[makeroom]{cancel}
\usepackage{stmaryrd}
\usepackage{amsmath,bm}
\usepackage{amssymb}
\usepackage{mathtools, cuted}
\usepackage{kantlipsum,setspace}

\theoremstyle{plain}
\newtheorem{theorem}{Theorem}[section]
\newtheorem{lemma}[theorem]{Lemma}
\newtheorem*{lemma*}{Lemma}

\newtheorem{cor}{Corollary}
\newtheorem*{cor*}{Corollary}

\theoremstyle{definition}
\newtheorem{defn}{Definition}[section]
\newtheorem*{defn*}{Definition}
\newtheorem{assump}{Assumption}

\theoremstyle{remark}
\newtheorem{rem}{Remark}
\newtheorem*{rem*}{Remark}

\title{Distributed Stochastic Non-Convex Optimization: Momentum-Based Variance Reduction}

\author{\IEEEauthorblockN{Prashant Khanduri$^{\dagger}$, Pranay Sharma$^{\dagger}$,  Swatantra Kafle$^{\dagger}$, Saikiran Bulusu$^{\dagger}$, \\
Ketan Rajawat$^{\ast}$, and Pramod K. Varshney$^{\dagger}$}\\
$^{\dagger}$Department of Electrical Engineering and Computer Science,
Syracuse University, Syracuse, NY\\
$^\ast$Department of Electrical Engineering,
IIT Kanpur, Kanpur, India.\\
Email: \{pkhandur, psharm04, skafle, sabulusu, varshney\}@syr.edu, ketan@iitk.ac.in}

\begin{document}
\maketitle
\begin{abstract}
In this work, we propose a distributed algorithm for stochastic non-convex optimization. We consider a worker-server architecture where a set of $K$ worker nodes (WNs) in collaboration with a server node (SN) jointly aim to minimize a global, potentially non-convex objective function. The objective function is assumed to be the sum of local objective functions available at each WN, with each node having access to only the stochastic samples of its local objective function. In contrast to the existing approaches, we employ a momentum based ``single loop" distributed algorithm which eliminates the need of computing large batch size gradients to achieve variance reduction. We propose two algorithms one with ``\emph{adaptive}" and the other with ``\emph{non-adaptive}" learning rates. We show that the proposed algorithms achieve the optimal computational complexity while attaining linear speedup with the number of WNs. Specifically, the algorithms reach an $\epsilon$-stationary point $x_a$ with $\mathbb{E}\| \nabla f(x_a) \| \leq \tilde{O}(K^{-1/3}T^{-1/2} + K^{-1/3}T^{-1/3})$ in $T$ iterations, thereby requiring $\tilde{O}(K^{-1} \epsilon^{-3})$ gradient computations at each WN. Moreover, our approach does not assume identical data distributions across WNs making the approach general enough for federated learning applications. 
\end{abstract}


\IEEEpeerreviewmaketitle
\section{Introduction}
\label{sec: Intro}
Unconstrained stochastic non-convex optimization problems aim to minimize a function $f: \mathbb{R}^d \to \mathbb{R}$ denoted as:
\begin{align}
\label{eq: CentralizedProb}
  \min_{x \in \mathbb{R}^d} \Big\{ f(x) := \mathbb{E}_{\xi \sim \mathcal{D}}[ f(x; \xi)] \Big\}.
\end{align}
where $f(x; \xi)$ represents a sample function (potentially non-convex), specifically, a sample of $f$ drawn from distribution  $\mathcal{D}$, i.e. $\xi \sim \mathcal{D}$. This setting is sometimes also referred to as the online setting \cite{Chang_Hong_Arxiv_2020}, where stochastic samples of the function $f$ are observed in a streaming setting. Non-convex stochastic optimization problems of the form given in \eqref{eq: CentralizedProb} cover a myriad of machine learning applications \cite{Bottou_SIAM_2018_Review}. Such problems often arise in training of deep neural networks \cite{goodfellow16book}, matrix completion \cite{chi19tsp}, principal component analysis \cite{netrapalli14nips, kang15icdm}, tensor decomposition \cite{ge15colt}, inference in graphical models \cite{xu17nips} and maximum likelihood estimation with hidden variables to name a few. Moreover, in the current age of Big-Data, the learning systems on a consistent basis have to draw inferences with large (potentially infinite) or streaming data samples. For such models, it is not always feasible to implement the algorithms and perform all the computations at a single central node \cite{Xing_Engineering_2016}. To alleviate this shortcoming and to speedup the computations, modern machine learning applications utilize the worker-server architecture, where multiple worker nodes (WNs) in collaboration with a server node (SN) jointly aim to solve \eqref{eq: CentralizedProb}. The worker-server architecture entails off-loading data as well as computations to the WNs which helps not only in speeding up the algorithms as well as providing some level of data privacy \cite{Leaute_JAIR_2013}. The SN acts as the central server responsible for sharing the parameters between the WNs. Assuming that there are $K$ WNs present in the network, problem \eqref{eq: CentralizedProb} for a distributed setup can be reformulated as:
\begin{align}
\label{eq: DecentralizedProb}
     \min_{x \in \mathbb{R}^d} ~& \bigg\{ f(x) := \frac{1}{K} \sum_{k = 1}^K  \mathbb{E}_{\xi \sim \mathcal{D}^{(k)}}[ f^{(k)}(x; \xi)]\bigg\}.
\end{align}
where $\mathcal{D}^{(k)}$ represents the distribution of the samples at the $k$-th WN for $k \in [K]$. Since the distributions $\mathcal{D}^{(k)}$ can be different across different WNs, \eqref{eq: DecentralizedProb} also models the popular federated learning models \cite{Konevcny_Arxiv_2016}. Finding a global solution of the non-convex problem \eqref{eq: CentralizedProb} (or \eqref{eq: DecentralizedProb}) is, in general, an NP-hard problem. Therefore, the solution methodologies for finding the minimizer of $f$ generally rely on iterative methods for finding an approximate solution. These methods are often designed to find one of the approximate stationary points $x_a$ (see Definition \ref{Def: StationaryPt}) of the function $f$ such that $\mathbb{E}[\nabla f(x_a)]$ is close to $0$. 

The goal of this work is to design such algorithms for solving \eqref{eq: DecentralizedProb} with the help of $K$ WNs and a SN, which utilizes only the stochastic gradient information to find an approximate stationary point of function $f$. The de facto standard algorithm used in large scale machine learning for such problems is stochastic gradient descent (SGD) \cite{Ghadimi_Siam_2013_SGD}, which can be written for problem \eqref{eq: CentralizedProb} as:
$$x_{t + 1} = x_t - \eta_t\nabla f(x_t; \xi_t)$$
where $\{x_t\}_{t=0}^T$ are the iterates returned from SGD,  $\nabla f(x_t; \xi_t)$ is the direction of the stochastic gradient at $t$th time with $\xi_t \sim \mathcal{D}$ and $\eta_t$ is the step size (sometimes also referred to as the learning rate). The convergence performance of the SGD is sensitive to the choice of the step-size \cite{Ward_Arxiv_2019_adagrad}  and it has been shown to converge only under restrictive conditions on the step-size \cite{Ghadimi_Siam_2013_SGD, Bottou_SIAM_2018_Review}. To remedy this sensitivity and make SGD robust to the choice of parameters ``\emph{adaptive}" SGD methods are often used where the step-size is chosen on the fly using stochastic gradient information of the current and the past samples \cite{Duchi_JMLR_2011_adaptive,Kingma_Arxiv_2014_adam,Ward_Arxiv_2019_adagrad,Tieleman_Hinton_Lecture_2012RMSProp,Cutkosky_NIPS2019}. In this work, we propose one such ``\emph{adaptive}" method for distributed non-convex stochastic optimization which utilizes the current stochastic gradient information to design the step-sizes. The proposed algorithm is a substantial extension of the centralized algorithm STORM proposed in \cite{Cutkosky_NIPS2019} for worker-server types of architecture. The major advantage of the proposed algorithm is that it achieves the optimal computational complexity (please see Definition \ref{Def: ComputationComplexity}) while achieving linear speed-up with the number of WNs without computing any batch gradients at each WN in contrast to the current state-of-the-art methods which almost exclusively rely on computing mega batch-size gradients at each WN to guarantee optimal convergence guarantees. Moreover, the proposed algorithm executes in a ``single loop" and achieves variance reduction by designing the descent direction based on momentum-based constructions \cite{Dinh_Arxiv_2019,Cutkosky_NIPS2019}. In contrast, the most popular variance reduction based methods utilize double loop structure to achieve variance reduction, with one loop designed for variance reduction and the other for speeding-up the algorithm \cite{Johnson_NIPS_2013, Reddi_ICML_2016,Fang_NIPS_2018_spider}. 
\subsection{Related Work and Contributions}
\subsubsection{Centralized Algorithms}
Due to the inherent variance of the stochastic gradients, the SGD proposed in \cite{Ghadimi_Siam_2013_SGD, Bottou_SIAM_2018_Review}, has been shown to converge to the $\epsilon$-stationary point at a modest rate of  $O(T^{-1/4})$ for non-convex smooth objective functions. This rate was improved to $O(T^{-3/10})$ in non-convex SVRG
\cite{Zhu_Hazan_ICML_2016, Reddi_ICML_2016} and in SCSG \cite{Lei_Jordan_SCSG}, by using the ideas initially developed in SVRG \cite{Johnson_NIPS_2013} for strongly convex objectives. Based on similar ideas, the convergence guarantees of stochastic gradient based methods for smooth non-convex objective functions were further improved to $O(T^{-1/3})$ in SPIDER \cite{Fang_NIPS_2018_spider}, SpiderBOOST \cite{Wang_NIPS_2019_SpiderBoost} and Nested SVRG proposed in \cite{Zhou_NIPS_2018}. This convergence rate was later shown to be optimal in \cite{Arjevani_Carmon_2019_LowerBounds}. As discussed earlier, these algorithms relied on a double loop structure to achieve variance reduction by computing large batch size gradients and, thereby, improving upon the performance of SGD methods. To circumvent the need for computing these large batch gradients, Hybrid-SGD \cite{Dinh_Arxiv_2019} and STORM \cite{Cutkosky_NIPS2019} algorithms were developed recently which execute in a single loop and achieve variance reduction using momentum based gradient updates. These algorithms were shown to achieve the same optimal convergence rate of $O(T^{-1/3})$ (up to a logarithmic factor in \cite{Cutkosky_NIPS2019}) while computing only the stochastic gradients. Note that Hybrid-SGD in \cite{Dinh_Arxiv_2019} requires computation of one mega batch size gradient compared to \cite{Zhu_Hazan_ICML_2016, Reddi_ICML_2016, Lei_Jordan_SCSG,Fang_NIPS_2018_spider, Zhou_NIPS_2018, Wang_NIPS_2019_SpiderBoost}, where mega batch size gradients are computed for each outer loop. 

Moreover, as pointed out earlier, SGD based methods are not only sensitive to the choice of the parameters but are also convergent only under restrictive assumptions. To robustify SGD based algorithms, adaptive algorithms like AdaGrad, ADAM and RMSProp \cite{Duchi_JMLR_2011_adaptive, Kingma_Arxiv_2014_adam, Tieleman_Hinton_Lecture_2012RMSProp} have been proposed where the step size naturally adapts to the past and present stochastic gradients of the objective function. In the non-convex setting, these methods have been extended in 
\cite{Ward_Arxiv_2019_adagrad, Li_Orabona_Arxiv_2018_convergence, Reddi_Arxiv_2019_Adam} and more recently in \cite{Cutkosky_NIPS2019}. In contrast to other methods, STORM proposed in \cite{Cutkosky_NIPS2019} has shown that the state-of-the-art convergence guarantees can be achieved even with adaptive methods when combined with momentum based techniques. 
\subsubsection{Distributed Algorithms}
In distributed networks, linear speedup implies that the total IFO calls (please see Definition \ref{Def: ComputationComplexity}) required to achieve an $\epsilon$-stationary point (please see Definition \ref{Def: StationaryPt}) at each WN are reduced by a factor equal to the number of WNs, $K$, present in the network compared to a centralized system. In distributed setups, primary goal is to achieve linear speedup with the number of WNs present in the network. Distributed (also referred to as parallel) versions of SGD for minimizing non-convex objectives was initially proposed in \cite{Li_Smola_NIPS_2014_CommunicationEfficient,Dean_NIPS_2012large}. To further reduce the communication complexity, restarted versions of SGD for non-convex problems, where the communication was performed less frequently was proposed in \cite{Yu_Zhu_2018parallel, Wang_Joshi_Arxiv_2018cooperative, Yu_Jin_Arxiv_2019linear}. In contrast to \cite{Yu_Zhu_2018parallel} and \cite{Wang_Joshi_Arxiv_2018cooperative}, the authors of \cite{Yu_Jin_Arxiv_2019linear} employed momentum based techniques to further improve the communication efficiency of distributed SGD based algorithms. All the above algorithms achieved a convergence rate of $O(K^{-1/4}T^{-1/4})$, thereby achieving linear speed up with the number of WNs present in the network. In this work, we propose a momentum based SGD algorithm which improves upon the convergence rate of distributed as well as restarted versions of SGD and also achieves linear speedup with the number WNs. The proposed algorithm outperforms the existing algorithms by achieving variance reduction via momentum based descent direction construction.

As pointed out in \cite{Sharma_Arxiv_2019}, the literature on distributed variance reduction methods has almost exclusively focused on convex and strongly convex problems with a few exceptions including \cite{Sharma_Arxiv_2019} and \cite{Khanduri_Arxiv_2019}. In \cite{Khanduri_Arxiv_2019}, a robust distributed algorithm was developed with convergence rate of $O(K^{-1/5}T^{-3/10})$ where the responsibility of computing the mega size batch gradients was given to the WNs. This convergence rate was uniformly improved to $O(K^{-1/3}T^{-1/3})$ in PR-SPIDER proposed in \cite{Sharma_Arxiv_2019}, which is not only optimal \cite{Arjevani_Carmon_2019_LowerBounds}, but is also shown to achieve state-of-the-art communication complexity. However, similar to the centralized algorithms, this convergence performance was achieved at the expense of computing mega batch size gradients. Moreover, as pointed out earlier, we utilize the ideas developed in STORM proposed in \cite{Cutkosky_NIPS2019}, and, therefore, design an adaptive algorithm which utilizes momentum based techniques to achieve variance reduction. Specifically, in this work we propose the first distributed algorithm for stochastic non-convex optimization which:
\begin{itemize}
\item Proposes a single-loop algorithm which eliminates the need for computing stochastic gradients over large batches. The algorithm computes stochastic gradients only to update the descent direction. 
    \item  The proposed algorithm is shown to achieve linear speed up with the number of WNs in the network. 
    \item We first propose the algorithm which uses adaptive step sizes to achieve variance reduction. Then we propose a special case of the adaptive algorithm with non-adaptive step sizes, which requires a constant order less communication compared to the adaptive algorithm while providing the same convergence guarantees.  
\end{itemize}
\subsection{Paper Organization}
In Section \ref{sec: Problem}, we discuss the problem along with the corresponding assumptions and definitions. In Section \ref{sec: AD-STORM}, we first introduce the Adaptive algorithm ``AD-STORM" and then in Section \ref{sec: AD-STORM_Convergence} we present its convergence guarantees. Then in Section \ref{sec: D-STORM}, we propose the non-adaptive version of the algorithm ``D-STORM" along with its associated convergence guarantees. 
Finally, in Section \ref{sec: Conclusion} we conclude the paper. For improved readability, some of the proofs and results are provided in the Appendix. Below, we describe the notations used in the paper. 

\subsection{Notations}
The expected value of a random variable $X$ is denoted by $\mathbb{E}[X]$ and its conditional expectation conditioned on an Event $A$ is denoted as $\mathbb{E}[X| \text{Event}~A]$. We denote by $\mathbb{R}$ (and $\mathbb{R}^d$) the set of real numbers (and $d$-dimensional real-valued vectors). The set of natural numbers is denoted by $\mathbb{N}$. Given a positive integer $K \in \mathbb{N}$, $[K] \triangleq \{1,2, \ldots, K\}$. Throughout the manuscript, $\| \cdot \|$ denotes the $\ell_2$-norm and $\langle \cdot, \cdot \rangle$ the Euclidean inner product.
 
\section{Problem Formulation}
\label{sec: Problem}
The objective is to minimize $f: \mathbb{R}^d \to \mathbb{R}$
\begin{align*}
     \min_{x \in \mathbb{R}^d} ~& \bigg\{ f(x) := \frac{1}{K} \sum_{k = 1}^K f^{(k)}(x) \bigg\},
\end{align*}
where $K$ is the number of WNs. The $K$ functions, $f^{(k)}: \mathbb{R}^d \to \mathbb{R}$, are distributed among $K$ WNs with each node having access to stochastic samples of the locally available function. The function at the $k$th node is given as
\begin{align*}
    f^{(k)}(x) := \mathbb{E}_{\xi \sim \mathcal{D}^{(k)}}[ f^{(k)}(x; \xi)].
\end{align*}
where $\mathcal{D}^{(k)}$ represents the distribution of the samples at the $k$th WN. 
\begin{assump}[Gradient Lipschitz Continuity]
\label{Ass: Lip_Smoothness}
All the functions $f^{(k)}(\cdot, \xi^{(k)})$ with $\xi^{(k)} \sim \mathcal{D}^{(k)}$, for all $k \in [K]$ are assumed to be $L$-smooth, i.e.,
\begin{align*}
    \mathbb{E} \| \nabla f^{(k)} (x ; \xi^{(k)}) -  \nabla f^{(k)} (y ; \xi^{(k)}) \| \leq L  \| x - y \|~~~\text{for all}~x,y \in \mathbb{R}^d.
\end{align*}
\end{assump}

\begin{assump}[Unbiased Bounded Gradient and Variance Bound]
\label{Ass: Unbiased_Var_Grad}
For all the functions $f^{(k)}(x , \xi^{(k)})$ with $\xi^{(k)} \sim \mathcal{D}^{(k)}$, for all $k \in [K]$ we have:
\begin{enumerate}
    \item Unbiased gradient: We assume that each stochastic gradients computed at each WN is an unbiased estimate of the corresponding function of the WN, i.e., 
    $$\mathbb{E}[\nabla f^{(k)} (x ; \xi^{(k)})] = \nabla f^{(k)} (x).$$
    Moreover, we assume that each node chooses samples $\xi^{(k)}$ independently across all $k \in [K]$. 
    
    \item Variance bound: We assume that the variance of the stochastic gradients computed at each node is universally bounded as:
    $$\mathbb{E}\| \nabla f^{(k)} (x ; \xi^{(k)}) - \nabla f^{(k)} (x) \|^2 \leq \sigma^2.$$
    \item Bounded Gradient: Finally, to design the adaptive algorithm AD-STORM, we additionally need the stochastic gradients to be bounded at each WN, i.e.,
    $$\| \nabla f^{(k)}(x; \xi^{(k)}) \| \leq G.$$
\end{enumerate}
\end{assump}

\begin{defn}[$\epsilon$-stationary Point]
\label{Def: StationaryPt}
  A point $x$ is called $\epsilon$-stationary if $\| \nabla f(x) \| \leq \epsilon$. Moreover, a stochastic algorithm is said to achieve an $\epsilon$-stationary point in $t$ iterations if $\mathbb{E}[\| \nabla f(x_t) \|] \leq \epsilon$, where the expectation is over the randomness of the algorithm until time instant $t$.
	\end{defn}
	
	\begin{defn}[Computation complexity]
\label{Def: ComputationComplexity}
We assume an Incremental First-order Oracle (IFO) framework, where given a sample $\xi^{(k)}$ at the $k$th node and iterate $x$, the oracle returns $(f^{(k)}(x; \xi^{(k)}), \nabla f^{(k)}(x ; \xi^{(k)}) )$. Each access to the oracle is counted as a single IFO operation. We measure the computational complexity in terms of the total number of calls to the IFO each WN makes to achieve an $\epsilon$-stationary point given in Definition \ref{Def: StationaryPt}.
	\end{defn}

	\section{Adaptive Distributed Algorithm: AD-STORM}
	\label{sec: AD-STORM}
	\subsection{A Biased Gradient Estimator}
In this work, we use the gradient estimator similar to the one employed in \cite{Cutkosky_NIPS2019} and \cite{Dinh_Arxiv_2019}. The proposed gradient estimator at each node takes a convex combination of the popular SARAH estimator \cite{Nguyen_ICML_2017_SARAH} and the unbiased SGD gradient estimator \cite{Ghadimi_Siam_2013_SGD}. We can write the gradient estimator at node $k \in [K]$ at time instant $t + 1$ as:
\begin{align*}
 d_{t+1}^{(k)} & = a_{t+1} \underbrace{\nabla f^{(k)}(x_{t+1}^{(k)} ; \xi_{t+1}^{(k)})}_{\text{SGD}} + (1 - a_{t+1})     \underbrace{ \big(  d_{t}^{(k)} + \nabla f^{(k)}(x_{t+1}^{(k)} ; \xi_{t+1}^{(k)}) - \nabla f^{(k)}(x_{t}^{(k)} ; \xi_{t+1}^{(k)})\big)}_{\text{SARAH}}\\
 & = \nabla f^{(k)}(x_{t+1}^{(k)} ; \xi_{t+1}^{(k)}) + (1 - a_{t+1})      \big(  d_{t}^{(k)} - \nabla f^{(k)}(x_{t}^{(k)} ; \xi_{t+1}^{(k)}) \big).
\end{align*}
with the parameter $a_{t+1} \leq 1$ as the momentum parameter. Note that similar estimators have been previously used in \cite{Cutkosky_NIPS2019} and \cite{Dinh_Arxiv_2019} in centralized settings where all the data is available at a single node. In contrast to \cite{Sharma_Arxiv_2019, Khanduri_Arxiv_2019, Fang_NIPS_2018_spider, Wang_NIPS_2019_SpiderBoost, Nguyen_ICML_2017_SARAH, Lei_Jordan_SCSG, Zhu_Hazan_ICML_2016}, the gradient estimator proposed above does not require the computation of large batch-size gradients to achieve optimal computational complexity. In contrast, the proposed gradient estimator above computes two stochastic gradients, $\nabla f^{(k)}(x_{t+1}^{(k)}; \xi_{t+1}^{(k)})$ and $\nabla f^{(k)}(x_{t}^{(k)}; \xi_{t+1}^{(k)})$ at each node for $k \in [K]$ at each time instant. A similar estimator used in a centralized setting for the STORM algorithm was shown to achieve optimal convergence guarantees in \cite{Cutkosky_NIPS2019} recently. Next, we discuss the algorithm.
	
\subsection{Algorithm: AD-STORM}
In this section, we propose an adaptive distributed algorithm AD-STORM based on STORM proposed in \cite{Cutkosky_NIPS2019}.

The pseudo code of AD-STORM is presented in Algorithm \ref{Algo_AD-STORM}. After specifying the required parameters, we initialize the algorithm with the same initial iterate, $x_1^{(k)} = \bar{x}_1$, at each WN. Also, each node uses the same initial descent direction, $d_1^{(k)} = \bar{d}_1$, which is constructed using unbiased estimates of the gradients at individual WNs, $\nabla f^{(k)}(x_1^{(k)}; \xi_1^{(k)})$, for all $k \in [K]$ with $\xi_1^{(k)} \sim \mathcal{D}^{(k)}$. Then in Step 5, the WNs share the norm of stochastic gradients computed at each WN with the SN, which is a scalar value. The SN computes aggregated statistic $\bar{G}_t^2$ and sends it back to the WNs. Based on the received statistic $\bar{G}_t^2$, each WN updates the parameter $w_t$ in Step 7 and designs its corresponding step-size $\eta_t$ in Step 8 (note that alternatively Steps 7 and 8 can be preformed at the SN and $\eta_t$ can be send to the WNs instead of $\bar{G}_t^2$). In Step 9 of the algorithm, the individual WNs update their corresponding iterate, $x_t^{(k)}$. Note at this stage that as each node had the same initial iterate, $\bar{x}_1$, and each node uses the same descent direction the updated iterate, $x_{t+1}^{k}$, will be same across all $k \in [K]$. Therefore, we denote  $x_{t+1}^{k} = \bar{x}_{t+1}$ for $t \in [T]$. Next, the momentum parameter, $a_{t+1}$, is updated in Step 10. Using the updated iterate and the momentum parameter, each WN in Step 11 updates the descent direction and forwards it to the SN. The SN aggregates the received descent directions and sends them back to the WNs in Step 12. The process repeats until convergence. 

\begin{rem}
The parameter $w_t$ for $t \in [T]$ is a non-increasing function of time, in contrast to the STORM algorithm proposed in \cite{Cutkosky_NIPS2019}, where the parameter $w_t$ remained fixed over time. Moreover, the parameter $\bar{\kappa}$ given in Theorem \ref{Thm: AD_Convergence_Main} is larger than the one used in STORM. In conclusion, the distributed architecture allows us to choose the step-sizes $\eta_t$ that are larger compared to the centralized case. This choice of increased step sizes and the corresponding analysis shows that AD-STORM (and D-STORM in Section \ref{sec: D-STORM}) is capable of achieving linear speed up with the number of WNs, $K$, present in the network.  
\end{rem}

Next, we provide the convergence guarantees for AD-STORM. 
	\begin{algorithm}[t]
\caption{Adaptive Distributed STORM - AD-STORM}
\label{Algo_AD-STORM}
\begin{algorithmic}[1]
	\State{\textbf{Input}: Parameters: $\bar{\kappa}$, $w_0$ and $c$}
	\State{\textbf{Initialize}:  Iterate $x_1^{(k)} = \bar{x}_1$, descent direction $d_1^{(k)} = \bar{d}_1= \frac{1}{K}\sum_{k=1}^K \nabla f^{(k)}(x_1^{(k)} ; \xi_1^{(k)})$ for all $k \in [K]$, step size $\displaystyle \eta_0^{(k)} = \frac{\bar{\kappa}}{w_0^{1/3}}$.}
	\For{$t = 1$ to $T$}
    	\For{$k = 1$ to $K$}
    	\State{$g_t^{(k)} = \|\nabla f^{(k)}(x_t^{(k)}; \xi_t^{(k)}) \|$\hfill $\to$ Forward $g_t^{(k)}$ to SN}
    		\State{$\bar{G}_t^2  = \frac{1}{K} \sum_{k=1}^K (g_t^{(k)})^2 $\hfill $\to$ Receive $\bar{G}_t^2$ from SN}
    		\State{$w_t = \max \big\{2 G^2, \bar{\kappa}^3L^3 - \sum_{i = 1}^t \bar{G}_i^2,  {\bar{\kappa}^3 c^3}/{L^3} \big\}$}
    	   \State{$\eta_t = \frac{\bar{\kappa}}{\big(w_t + \sum_{i = 1}^t \bar{G}_i^2  \big)^{1/3}}$}
        	\State{$x_{t+1}^{(k)} =  x_{t}^{(k)} - \eta_{t} d_{t}^{(k)}$}
        		\State{$a_{t+1} = c \eta_t^2$}
        		\State{$d_{t+1}^{(k)} = \nabla f^{(k)}(x_{t+1}^{(k)} ; \xi_{t+1}^{(k)}) + (1 - a_{t+1})      \big(  d_t^{(k)} - \nabla f^{(k)}(x_t^{(k)} ; \xi_{t+1}^{(k)}) \big)$ \hfill $\to$ Forward $d_t^{(k)}$ to SN}
        			\State{$d_{t+1}^{(k)} = \bar{d}_{t+1} = \frac{1}{K} \sum_{k=1}^K d_{t+1}^{(k)}$ \hfill $\to$ Receive $\bar{d}_t$ from SN}
	    \EndFor
	\EndFor
\end{algorithmic}
\end{algorithm}

\section{Analysis: AD-STORM}
\label{sec: AD-STORM_Convergence}
The goal of Algorithm \ref{Algo_AD-STORM} is to guarantee:
\begin{align*}
   \mathbb{E} \big\| \nabla f (\bar{x_t}) \big\| =  \mathbb{E} \bigg\| \frac{1}{K} \sum_{k=1}^K \nabla f^{(k)} (\bar{x_t}) \bigg\|  \leq \epsilon ~~\text{where}~~\bar{x}_t = \frac{1}{K}\sum_{k = 1}^K x_t^{(k)}.
\end{align*}
which is the definition of the $\epsilon$-stationary point. 

In this section, we present the analysis of AD-STORM and present its convergence guarantees. As a starting point, we first need the basic descent lemma as proved below. 
\begin{lemma}[Descent Lemma]
\label{Lem: AD_DescentLemma}
For $\eta_t \leq \frac{1}{L}$ and $\bar{e}_t = \bar{d}_t - \nabla f(\bar{x}_{t})$, we have:
\begin{align*}
  \mathbb{E} [f(\bar{x}_{t + 1})]   \leq   \mathbb{E}  \bigg[ f(\bar{x}_{t })  - \frac{\eta_t}{2}   \|\nabla f(\bar{x}_t) \|^2 + \frac{\eta_t}{2}   \| \bar{e}_t  \|^2 \bigg]. 
\end{align*}
\end{lemma}
\begin{proof}
Using the smoothness of $f$ we have:
\begin{align}
  \mathbb{E} [ f(\bar{x}_{t + 1})] & \leq \mathbb{E} \bigg[ f(\bar{x}_{t }) + \langle \nabla f(\bar{x}_{t}),  \bar{x}_{t + 1} - \bar{x}_{t}\rangle + \frac{L}{2} \| \bar{x}_{t + 1} - \bar{x}_{t } \|^2 \bigg] \nonumber\\
    &  \overset{(a)}{=} \mathbb{E}\bigg[f(\bar{x}_{t}) - \eta_t \langle \nabla f(\bar{x}_{t}),  \bar{d}_t \rangle + \frac{\eta_t^2 L}{2} \| \bar{d}_{t}  \|^2 \bigg] \nonumber\\
     & = \mathbb{E}\bigg[f(\bar{x}_{t}) - \eta_t  \| \bar{d}_{t}  \|^2 + \eta_t \langle \bar{d}_t - \nabla f(\bar{x}_{t}),  \bar{d}_t \rangle + \frac{\eta_t^2 L}{2} \| \bar{d}_{t}  \|^2\bigg] \nonumber\\
       & \overset{(b)}{=} \mathbb{E}\bigg[f(\bar{x}_{t}) - \eta_t  \| \bar{d}_{t}  \|^2 + \frac{\eta_t}{2} \| \bar{d}_t - \nabla f(\bar{x}_{t})\|^2 + \frac{\eta_t}{2} \|  \bar{d}_t \|^2 - \frac{\eta_t}{2} \| \nabla f(\bar{x}_t) \|^2 + \frac{\eta_t^2 L}{2} \| \bar{d}_{t}  \|^2 \bigg] \nonumber\\
       & =    \mathbb{E}\bigg[ f(\bar{x}_{t }) -  \left( \frac{\eta_t}{2} - \frac{\eta_t^2 L}{2} \right)  \| \bar{d}_{t}  \|^2 - \frac{\eta_t}{2} \|\nabla f(\bar{x}_t) \|^2 + \frac{\eta_t}{2}  \| \bar{d}_t - \nabla f(\bar{x}_{t}) \|^2 \bigg] \nonumber \\  
       & \overset{(c)}{\leq}  \mathbb{E}\bigg[ f(\bar{x}_{t })  - \frac{\eta_t}{2} \|\nabla f(\bar{x}_t) \|^2 + \frac{\eta_t}{2}  \| \bar{d}_t - \nabla f(\bar{x}_{t}) \|^2\bigg].    \label{eq: AD_Smoothness_1st}
\end{align}
where $(a)$ follows from the update given in Step 9 of Algorithm \ref{Algo_AD-STORM}, $(b)$ follows from the relation $\left\langle a,b \right\rangle = \frac{1}{2} \|a\|^2 + \frac{1}{2} \|b\|^2 - \frac{1}{2} \|a-b\|^2$ and $(c)$ follows from the choice $\eta_t \leq \frac{1}{L}$. 

Finally, using the notation $\bar{e}_t = \bar{d}_t - \nabla f(\bar{x}_{t})$ we have the proof.
\end{proof}
Now, to study the contraction of the average gradient error, $\bar{e}_t$, we analyze the contraction properties of $\| \bar{e}_t \|$ in the next lemma. Specifically, we analyze how the term ${\| \bar{e}_t \|^2}/{\eta_{t-1}}$ contracts across time.  Note that this construction helps us analyze the non-trivial potential function defined later in the section. We name the lemma as Error contraction lemma.   
\begin{lemma}[Error Contraction] 
\label{Lem: AD_ErrorContractionLemma}
The error term $\bar{e}_t$ from Algorithm \ref{Algo_AD-STORM} satisfies the following:
\begin{align*}
    \mathbb{E} \bigg[ \frac{\| \bar{e}_t \|^2}{\eta_{t-1}} \bigg] \leq  \mathbb{E} \bigg[ \bigg( (1 - a_t)^2  + \frac{4 (1 - a_t)^2 L^2 \eta_{t-1}^2}{K} \bigg) \frac{\|  \bar{e}_{t-1} \|^2}{\eta_{t-1}} + \frac{4 (1 - a_t)^2 L^2 \eta_{t-1}}{K}  \|  \nabla f(\bar{x}_{t-1}) \|^2 + \frac{2c^2 \eta_{t-1}^3 \bar{G}_t^2}{K} \bigg].
\end{align*}
\end{lemma}
\begin{proof}
Using the definition of $\bar{e}_t$ we have
\begin{align}
   & \mathbb{E} \bigg[ \frac{\| \bar{e}_t \|^2}{\eta_{t-1}} \bigg]  =  \mathbb{E} \bigg[ \frac{ \big\| \bar{d}_t - \nabla f(\bar{x}_t)  \big\|^2}{\eta_{t-1}} \bigg] \nonumber \\
    & \overset{(a)}{=} \mathbb{E} \bigg[ \frac{1}{\eta_{t-1}} \bigg\| \frac{1}{K} \sum_{k = 1}^K \nabla f^{(k)}(x^{(k)}_{t};\xi_t^{(k)}) + (1 - a_t)\bigg( \bar{d}_{t-1} - \frac{1}{K} \sum_{k = 1}^K \nabla f^{(k)}(x^{(k)}_{t - 1}; \xi_t^{(k)})\bigg) - \nabla f(\bar{x}_t)   \bigg\|^2 \bigg] \nonumber\\
    &  \overset{(b)}{=} \mathbb{E}  \bigg[ \frac{1}{\eta_{t-1}}  \bigg\| \frac{1}{K} \sum_{k = 1}^K \Big[\big( \nabla f^{(k)}(\bar{x}_t;\xi_t^{(k)})  -  \nabla f^{(k)}(\bar{x}_t) \big)    - (1 - a_t) \big( \nabla f^{(k)}(\bar{x}_{t - 1}; \xi_t^{(k)}) - \nabla f^{(k)}(\bar{x}_{t - 1})\big) \Big] + (1 - a_t) \bar{e}_{t-1}   \bigg\|^2 \bigg]\nonumber \\
    &   \overset{(c)}{=} \mathbb{E} \Bigg[  \frac{(1 - a_t)^2 \| \bar{e}_{t-1}\|^2}{\eta_{t-1}}  + \frac{1}{\eta_{t-1}K^2}\bigg\|\sum_{k = 1}^K \Big[\big( \nabla f^{(k)}(\bar{x}_t;\xi_t^{(k)})  -  \nabla f^{(k)}(\bar{x}_t) \big)    - (1 - a_t) \big( \nabla f^{(k)}(\bar{x}_{t - 1}; \xi_t^{(k)}) - \nabla f^{(k)}(\bar{x}_{t - 1})\big) \Big] \bigg\|^2  \nonumber\\
    & \qquad +\frac{2}{\eta_{t-1}}  \bigg\langle  (1 - a_t) \bar{e}_{t-1}, \frac{1}{K}\sum_{k = 1}^K \Big[\big( \nabla f^{(k)}(\bar{x}_t;\xi_t^{(k)})  -  \nabla f^{(k)}(\bar{x}_t) \big)    - (1 - a_t) \big( \nabla f^{(k)}(\bar{x}_{t - 1}; \xi_t^{(k)}) - \nabla f^{(k)}(\bar{x}_{t - 1})\big) \Big]   \bigg\rangle \Bigg] \nonumber \\
    &  \overset{(d)}{=} \mathbb{E} \Bigg[ \frac{(1 - a_t)^2 \| \bar{e}_{t-1}\|^2}{\eta_{t-1}}  + \frac{1}{\eta_{t-1}K^2 } \sum_{k = 1}^K  \big\| \big( \nabla f^{(k)}(\bar{x}_t;\xi_t^{(k)})  -  \nabla f^{(k)}(\bar{x}_t) \big)    - (1 - a_t) \big( \nabla f^{(k)}(\bar{x}_{t - 1}; \xi_t^{(k)}) - \nabla f^{(k)}(\bar{x}_{t - 1})\big)  \big\|^2 \nonumber\\
    & \qquad \qquad + \frac{1}{\eta_{t-1}K^2} \sum_{k,l \in [K],k \neq l} \bigg\langle   \big( \nabla f^{(k)}(\bar{x}_t;\xi_t^{(k)})  -  \nabla f^{(k)}(\bar{x}_t) \big)    - (1 - a_t) \big( \nabla f^{(k)}(\bar{x}_{t - 1}; \xi_t^{(k)}) - \nabla f^{(k)}(\bar{x}_{t - 1})\big), \nonumber\\
   & \qquad \qquad \qquad \qquad  \qquad  \quad      \qquad \quad \big( \nabla f^{(l)}(\bar{x}_t;\xi_t^{(l)})  -  \nabla f^{(l)}(\bar{x}_t) \big)    - (1 - a_t) \big( \nabla f^{(l)}(\bar{x}_{t - 1}; \xi_t^{(l)}) - \nabla f^{(l)}(\bar{x}_{t - 1})\big) \bigg\rangle  \Bigg]\nonumber\\
   & \overset{(e)}{=} \mathbb{E} \bigg[ \frac{(1 - a_t)^2 \| \bar{e}_{t-1}\|^2}{\eta_{t-1}}  + \frac{1}{\eta_{t-1} K^2} \sum_{k = 1}^K \big\| \big( \nabla f^{(k)}(\bar{x}_{t};\xi_t^{(k)})  -  \nabla f^{(k)}(\bar{x}_t) \big)    - (1 - a_t) \big( \nabla f^{(k)}(\bar{x}_{t - 1}; \xi_t^{(k)}) - \nabla f^{(k)}(\bar{x}_{t - 1})\big)  \big\|^2 \bigg].
   \label{Eq: AD_Error_Contraction1}
\end{align}
where $(a)$ follows from the definition of the descent direction $d_t^{(k)}$ given in Step 11 of Algorithm \ref{Algo_AD-STORM}, $(b)$ follows from adding and subtracting $(1- a_t) \frac{1}{K} \sum_{k=1}^K \nabla f^{(k)}(\bar{x}_{t-1}) (= (1 - a_t) \nabla f(\bar{x}_{t-1})$), the fact that $x_{t}^{(k)} = \bar{x}_t$ for all $t \in [T]$ (please see Step 9 of Algorithm \ref{Algo_AD-STORM}) and using the definition of $\bar{e}_{t-1}$, $(c)$ follows from expanding the norm using inner products, $(d)$ follows from the application of Lemma \ref{Lem: AD_InnerProduct_e_t_Grad} and again expanding the norm using the inner products across $k \in [K]$, finally $(e)$ results from the usage of Lemma \ref{Lem: AD_InnerProd_AcrossNodes}.

Now let us consider the 2nd term of \eqref{Eq: AD_Error_Contraction1} above we have:
\begin{align}
  &  \mathbb{E} \bigg[ \frac{1}{\eta_{t-1}K^2} \sum_{k = 1}^K  \big\| \big( \nabla f^{(k)}(\bar{x}_{t};\xi_t^{(k)})  -  \nabla f^{(k)}(\bar{x}_{t}) \big)    - (1 - a_t) \big( \nabla f^{(k)}(\bar{x}_{t - 1}; \xi_t^{(k)}) - \nabla f^{(k)}(\bar{x}_{t - 1})\big)  \big\|^2 \bigg] \nonumber\\
    & =  \mathbb{E} \bigg[ \frac{1}{\eta_{t-1}K^2} \sum_{k = 1}^K \big\| (1 -a_t) \Big[ \big( \nabla f^{(k)}(\bar{x}_{t};\xi_t^{(k)})  -  \nabla f^{(k)}(\bar{x}_{t}) \big)    - \big( \nabla f^{(k)}(\bar{x}_{t - 1}; \xi_t^{(k)}) - \nabla f^{(k)}(\bar{x}_{t - 1})\big) \Big] \nonumber \\
    & \qquad \qquad \qquad  \qquad \qquad \qquad \qquad  \qquad \qquad \qquad \qquad  \qquad  + a_t   \big( \nabla f^{(k)}(\bar{x}_{t}; \xi_t^{(k)}) - \nabla f^{(k)}(\bar{x}_{t})\big)      \big\|^2 \bigg] \nonumber\\
    & \overset{(a)}{\leq} \mathbb{E} \bigg[ \frac{2 (1 - a_t)^2}{\eta_{t-1}K^2} \sum_{k=1}^K  \big\| \big( \nabla f^{(k)}(\bar{x}_{t};\xi_t^{(k)})  -  \nabla f^{(k)}(\bar{x}_{t - 1}; \xi_t^{(k)}) \big)    -  \big(\nabla f^{(k)}(\bar{x}_{t}) - \nabla f^{(k)}(\bar{x}_{t - 1})\big) \big\|^2 \bigg] \nonumber\\
    & \qquad \qquad \qquad  \qquad \qquad \qquad \qquad \qquad \qquad \qquad   + \mathbb{E}\bigg[ \frac{2 a_t^2}{\eta_{t-1}K^2} \sum_{k=1}^K \big\|    \nabla f^{(k)}(\bar{x}_{t}; \xi_t^{(k)}) - \nabla f^{(k)}(\bar{x}_{t})    \big\|^2 \bigg]\nonumber\\
    & \overset{(b)}{\leq} \mathbb{E} \bigg[\frac{2 (1 - a_t)^2}{\eta_{t-1}K^2} \sum_{k=1}^K  \big\|   \nabla f^{(k)}(\bar{x}_t;\xi_t^{(k)})  -  \nabla f^{(k)}(\bar{x}_{t - 1}; \xi_t^{(k)}) \big\|^2 + \frac{2 c^2 \eta_{t-1}^3}{K^2}  \sum_{k=1}^K \big\|  \nabla f^{(k)}(\bar{x}_t; \xi_t^{(k)})\big\|^2 \bigg] \nonumber\\
    & \overset{(c)}{\leq} \mathbb{E} \bigg[ \frac{2 (1 - a_t)^2 L^2}{\eta_{t-1}K^2} \sum_{k=1}^K \| \bar{x}_t - \bar{x}_{t-1} \|^2 + \frac{2 c^2 \eta_{t-1}^3 \bar{G}_t^2}{K} \bigg] \nonumber \\
    & \overset{(d)}{=}  \mathbb{E}  \bigg[ \frac{2 (1 - a_t)^2 L^2 \eta_{t-1}}{K} \|  \bar{d}_{t-1} \|^2 + \frac{2 c^2 \eta_{t-1}^3 \bar{G}_t^2}{K} \bigg].
    \label{Eq: AD_Error_Contraction2}
\end{align}
where $(a)$ follows from using Lemma \ref{Lem: Norm_Ineq}, $(b)$ results from the use of Lemma \ref{Lem: AD_Mean_Variance_Bound} and the definition of $a_t$. Inequality $(c)$ follows from the Lipschitz continuity of the gradient given in Assumption \ref{Ass: Lip_Smoothness} and the definition of $\bar{G}_t^2$ given in Step 6 of Algorithm \ref{Algo_AD-STORM}. Finally, $(d)$ follows from the update Step 9 of Algorithm \ref{Algo_AD-STORM}.

Replacing, \eqref{Eq: AD_Error_Contraction2} in \eqref{Eq: AD_Error_Contraction1} we get:
\begin{align}
    \mathbb{E} \bigg[ \frac{\| \bar{e}_t \|^2}{\eta_{t-1}} \bigg] & \leq \mathbb{E} \bigg[ (1 - a_t)^2 \frac{\| \bar{e}_{t-1}\|^2}{\eta_{t-1}} + \frac{2 (1 - a_t)^2 L^2 \eta_{t-1}}{K}  \|  \bar{d}_{t-1} \|^2 + \frac{2 c^2 \eta_{t-1}^3 \bar{G}_t^2}{K} \bigg] \nonumber \\
     & \overset{(a)}{\leq} \mathbb{E} \bigg[ \bigg( (1 - a_t)^2  + \frac{4 (1 - a_t)^2 L^2 \eta_{t-1}^2}{K} \bigg) \frac{\|  \bar{e}_{t-1} \|^2}{\eta_{t-1}} + \frac{4 (1 - a_t)^2 L^2 \eta_{t-1}}{K} \|  \nabla f(\bar{x}_{t-1}) \|^2 + \frac{2 c^2 \eta_{t-1}^3 \bar{G}_t^2}{K} \bigg].
     \label{eq: AD_ErrorDescent}
\end{align}
where $(a)$ above follows by adding and subtracting $\nabla f(\bar{x}_{t-1})$ inside the norm $\| \bar{d}_{t-1}\|^2$ and using Lemma \ref{Lem: Norm_Ineq}. 

Therefore, we have the result.
\end{proof}
Before, proceeding further we first define the Lyapunov potential function to be used to get the main result of the paper. We define the Lyapunov potential function, $\Phi_t$, similar to \cite{Cutkosky_NIPS2019} as:
\begin{align}
\label{Eq: AD_PotentialFn}
\Phi_t = f(\bar{x}_t) + \frac{K}{48L^2 \eta_{t-1}} \| \bar{e}_t \|^2.
\end{align}

Using the above two lemmas, we can state the main result of the work in the next theorem. 
\begin{theorem}
\label{Thm: AD_Convergence_Main}
Under the assumptions given in Section \ref{sec: Problem} and for the choice of parameters:
\begin{enumerate}[(i)]
    \item For any $\displaystyle b^3 \geq \frac{2^{2/3}}{84}$.
    \item $\displaystyle \bar{\kappa} = \frac{b K^\alpha G^{2/3}}{L}$.
    \item $\displaystyle c = \frac{28L^2}{K} + \frac{2^{2/3} G^2}{3 L \bar{\kappa}^3} = L^2 \bigg( \frac{28}{K} + \frac{2^{2/3}}{3 b^3 K^{3\alpha}} \bigg) \overset{(i)}{\leq} \frac{56 L^2}{K}$.
    \item We have $\{w_t\}_{t=1}^T$ as
    \begin{align*}
w_t = \max \bigg\{2 G^2, \bar{\kappa}^3L^3 - \sum_{i = 1}^t \bar{G}_i^2,  \frac{\bar{\kappa}^3 c^3}{L^3} \bigg\} & = G^2 \max \bigg\{2, \frac{\bar{\kappa}^3L^3}{G^2} - \sum_{i = 1}^t \frac{\bar{G}_i^2}{G^2},  \frac{\bar{\kappa}^3 c^3}{G^2L^3} \bigg\} \\
&  \overset{(ii), (iii)}{\leq}  G^2 \max\bigg\{2 ,~ b^3 K^{3 \alpha} - \sum_{i = 1}^t \frac{\bar{G}_i^2}{G^2},  ~\frac{(56b)^3}{K^{3 - 3\alpha}} \bigg\}.
\end{align*}
\end{enumerate}
Then the algorithm AD-STORM satisfies:
\begin{align*}
    \mathbb{E} \Bigg[  \sqrt{\sum_{t=1}^T  \|  \nabla f(\bar{x}_{t}) \|^2 }  \Bigg]  \leq   \mathbb{E} \Bigg[\frac{M(\bar{\kappa}, w_0,c,\sigma) (w_T + G^2T)^{1/3} }{\bar{\kappa}} \bigg]^{1/2}.
\end{align*}
with $  M(\bar{\kappa}, w_0, c, \sigma)$ defined as
$$M(\bar{\kappa}, w_0, c, \sigma) = 3    \big(f(\bar{x}_1) - f(x^\ast)\big) +   \frac{w_0^{1/3}\sigma^2 }{20 L^2 \bar{\kappa}}  +  \frac{c^2 \bar{\kappa}^3}{10 L^2} \ln (T+2).$$
\end{theorem}

\begin{proof}
First, using Lemma \ref{Lem: AD_DescentLemma} and adding over $t = 1$ to $T$, we get:
\begin{align}
 \mathbb{E} [f(\bar{x}_{T + 1}) -   f(\bar{x}_{1})]  \leq     \mathbb{E} \bigg[ - \sum_{t=1}^T \frac{\eta_t}{2}  \|\nabla f(\bar{x}_t) \|^2 + \sum_{t=1}^T \frac{\eta_t}{2}  \| \bar{e}_t\|^2 \bigg].
    \label{eq: AD_Smoothness_Sum}
\end{align}
Now using Lemma \ref{Lem: AD_ErrorContractionLemma} we compute:
\begin{align}
 \mathbb{E} \bigg[   \frac{\|\bar{e}_{t+1} \|^2}{\eta_t}-  \frac{\|\bar{e}_{t} \|^2}{\eta_{t-1}} \bigg] & \leq \mathbb{E} \bigg[ \bigg( (1 - a_{t+1})^2  + \frac{4 (1 - a_{t+1})^2 L^2 \eta_{t}^2}{K} \bigg) \frac{\|  \bar{e}_{t} \|^2}{\eta_t} \nonumber\\
    & \qquad \qquad \qquad \quad + \frac{4 (1 - a_{t+1})^2 L^2 \eta_{t}}{K}  \|  \nabla f(\bar{x}_t) \|^2  + \frac{2c^2 \eta_t^3 \bar{G}_{t+1}^2}{K} -   \frac{\|\bar{e}_{t} \|^2}{\eta_{t-1}} \bigg] \nonumber\\
    & = \mathbb{E} \bigg[ \bigg( \eta_t^{-1} (1 - a_{t+1})^2 \bigg(1 + \frac{4 L^2 \eta_{t}^2}{K} \bigg) -  \eta_{t-1}^{-1} \bigg)\|\bar{e}_{t}\|^2  \nonumber\\
    & \qquad \qquad \qquad \qquad \qquad + \frac{4 (1 - a_{t+1})^2 L^2 \eta_{t}}{K}  \|  \nabla f(\bar{x}_t) \|^2  + \frac{2c^2 \eta_t^3 \bar{G}_{t+1}^2}{K} \bigg] \nonumber\\
    & \overset{(a)}{\leq}  \mathbb{E} \bigg[ \bigg( \eta_t^{-1} (1 - a_{t+1}) \bigg(1 + \frac{4 L^2 \eta_{t}^2}{K} \bigg) -  \eta_{t-1}^{-1} \bigg) \|  \bar{e}_{t} \|^2 + \frac{4 L^2 \eta_{t}}{K}  \|  \nabla f(\bar{x}_t) \|^2  + \frac{2c^2 \eta_t^3 \bar{G}_{t+1}^2}{K} \bigg].
    \label{eq: AD_Error_Potential}
\end{align}
where $(a)$ follows from the fact that $0 < 1 - a_t < 1$

Let us consider the coefficient of the first term of \eqref{eq: AD_Error_Potential}:
\begin{align}
     \eta_t^{-1} (1 - a_{t+1}) \bigg(1 + \frac{4 L^2 \eta_{t}^2}{K} \bigg) -  \eta_{t-1}^{-1} & = \eta_t^{-1} - \eta_{t-1}^{-1} + \eta_t^{-1} \bigg( \frac{4L^2 \eta_t^2}{K} - a_{t+1} \bigg( 1 + \frac{4L^2 \eta_t^2}{K} \bigg) \bigg) \nonumber\\
     & \overset{(a)}{\leq} \eta_t^{-1} - \eta_{t-1}^{-1} + \eta_t^{-1} \bigg( \frac{4L^2 \eta_t^2}{K} - a_{t+1}   \bigg) \nonumber\\
     & = \eta_t^{-1} - \eta_{t-1}^{-1} + \eta_t  \bigg( \frac{4L^2  }{K} - c  \bigg).
     \label{Eq: AD_Coefficient_e_t}
\end{align}
where inequality $(a)$ utilizes the fact that ${4 L^2 \eta_t^2}/{K}  > 0$.

First, considering $\eta_t^{-1} - \eta_{t-1}^{-1}$ in \eqref{Eq: AD_Coefficient_e_t} we have from the definition of $\eta_t$ in Algorithm \ref{Algo_AD-STORM}:
\begin{align}
    \eta_t^{-1} - \eta_{t-1}^{-1} & =  \frac{ \big(w_t + \sum_{i=1}^t \bar{G}_i^2 \big)^{1/3} - \big(w_{t-1} + \sum_{i=1}^{t-1} \bar{G}_i^2 \big)^{1/3} }{\bar{\kappa}} \nonumber \\
    & \overset{(a)}{\leq}     \frac{ \big(w_t + \sum_{i=1}^t \bar{G}_i^2 \big)^{1/3} - \big(w_{t} + \sum_{i=1}^{t-1} \bar{G}_i^2\big)^{1/3} }{\bar{\kappa}} \nonumber\\
    & \overset{(b)}{\leq} \frac{\bar{G}_t^2}{3 \bar{\kappa} \big(w_t + \sum_{i=1}^{t-1} \bar{G}_i^2 \big)^{2/3}} \nonumber\\
    & \overset{(c)}{\leq} \frac{\bar{G}_t^2}{3 \bar{\kappa} \big(w_t - G^2 + \sum_{i=1}^t \bar{G}_i^2 \big)^{2/3}} \nonumber\\
     & \overset{(d)}{\leq} \frac{\bar{G}_t^2}{3 \bar{\kappa} \big(w_t/2  + \sum_{i=1}^t \bar{G}_i^2 \big)^{2/3}} \nonumber\\
        & = \frac{2^{2/3} \bar{G}_t^2}{3 \bar{\kappa} \big(w_t  + 2 \sum_{i=1}^t \bar{G}_i^2  \big)^{2/3}} \nonumber \\
   & \leq \frac{2^{2/3}\bar{G}_t^2}{3 \bar{\kappa} \big(w_t + \sum_{i=1}^t \bar{G}_i^2 \big)^{2/3}} \overset{(e)}{\leq} \frac{2^{2/3}G^2 \bar{\kappa}^2}{3 \bar{\kappa}^3 \big(w_t + \sum_{i=1}^t \bar{G}_i^2  \big)^{2/3}}  \overset{(f)}{=} \frac{2^{2/3}G^2 }{3 \bar{\kappa}^3 } \eta_t^2 \overset{(g)}{\leq} \frac{2^{2/3}G^2}{3 L \bar{\kappa}^3 } \eta_t.
   \label{Eq: AD_Coefficient_e_t_1}
\end{align}
where inequality $(a)$ uses the fact that we have $w_t \leq w_{t-1}$ which follows from the definition of $w_t$ given in the statement $(iv)$ of Theorem \ref{Thm: AD_Convergence_Main} and $(b)$ follows from:
$$(x + y)^{1/3} - x^{1/3} \leq \frac{y}{3x^{2/3}}.$$
In $(c)$ we have used $\bar{G}_t^2 \leq G^2$ for all $t \in \mathbb{N}$ (please see \eqref{Eq: DA_Bar_G_t}), in inequality $(d)$ we have used the fact that $w_t \geq 2 G^2$, which follows from the definition of $w_t$ given in statement $(iv)$ of Theorem \ref{Thm: AD_Convergence_Main}. Inequality $(e)$ again uses $\bar{G}_t^2 \leq G^2$ for all $t \in \mathbb{N}$. Finally, in $(f)$ and $(g)$ we used the definition of $n_t$ given in Algorithm \ref{Algo_AD-STORM} and the fact that $\eta_t \leq 1/L$, respectively.

Now consider the term $\displaystyle \eta_t \bigg( \frac{4L^2  }{K} - c \bigg)$ in \eqref{Eq: AD_Coefficient_e_t}, since we have $\displaystyle c = \frac{28L^2}{K} + \frac{2^{2/3}G^2}{3 L\bar{\kappa}^3}$ we get:
\begin{align}
    \eta_t \bigg( \frac{4L^2  }{K} - c \bigg) = \eta_t \bigg( - \frac{24L^2}{K} - \frac{2^{2/3}G^2}{3 L\bar{\kappa}^3} \bigg)
    \label{Eq: AD_Coefficient_e_t_2}
\end{align}
Replacing \eqref{Eq: AD_Coefficient_e_t_1} and \eqref{Eq: AD_Coefficient_e_t_2} in \eqref{Eq: AD_Coefficient_e_t}, we get:
\begin{align}
    \eta_t^{-1} (1 - a_{t+1}) \bigg(1 + \frac{4 L^2 \eta_{t}^2}{K} \bigg) -  \eta_{t-1}^{-1}  \leq - \frac{24 L^2\eta_t}{K} .
    \label{Eq: AD_Coefficient_e_t_Bound}
\end{align}
Replacing \eqref{Eq: AD_Coefficient_e_t_Bound} in \eqref{eq: AD_Error_Potential}, we get:
\begin{align*}
   \mathbb{E}  \bigg[ \frac{\|\bar{e}_{t+1} \|^2}{\eta_t}-  \frac{\|\bar{e}_{t} \|^2}{\eta_{t-1}} \bigg] & \leq \mathbb{E} \bigg[ - \frac{24 L^2 \eta_t}{K}  \|  \bar{e}_{t} \|^2  + \frac{4  L^2 \eta_{t} }{K}  \|  \nabla f(\bar{x}_t) \|^2  + \frac{2c^2 \bar{G}_{t+1}^2 \eta_t^3}{K} \bigg].
\end{align*}
Now summing over $t$ and multiplying by $\displaystyle \frac{K}{48L^2}$ we get:
\begin{align}
   \frac{K}{48L^2} \sum_{t=1}^T \mathbb{E} \bigg[ \frac{\|\bar{e}_{t+1} \|^2}{\eta_t}-  \frac{\mathbb{E}\|\bar{e}_{t} \|^2}{\eta_{t-1}} \bigg] & \leq  \mathbb{E} \bigg[ \sum_{t=1}^T - \frac{\eta_t}{2} \|  \bar{e}_{t} \|^2  +  \sum_{t=1}^T \frac{\eta_{t} }{12}  \|  \nabla f(\bar{x}_t) \|^2  + \sum_{t=1}^T \frac{c^2  \bar{G}_{t+1}^2  \eta_t^3}{24 L^2}\bigg].
   \label{eq: AD_Half_Potential_Sum}
\end{align}
Finally, considering the last term of \eqref{eq: AD_Half_Potential_Sum} above and using the definition of $\eta_t$ from Algorithm \ref{Algo_AD-STORM} we have:
\begin{align}
  \sum_{t=1}^T \frac{c^2  \bar{G}_{t+1}^2  \eta_t^3 }{24 L^2} & = \sum_{t=1}^T \frac{c^2 \bar{\kappa}^3   \bar{G}_{t+1}^2 }{24 L^2 (w_t + \sum_{i=1}^t \bar{G}_i^2)} \nonumber\\
  & \overset{(a)}{\leq} \sum_{t=1}^T \frac{c^2 \bar{\kappa}^3    \bar{G}_{t+1}^2 }{24 L^2 (2 G^2 + \sum_{i=1}^t \bar{G}_i^2)} \nonumber\\
  & \overset{(b)}{\leq} \sum_{t=1}^T \frac{c^2 \bar{\kappa}^3    \bar{G}_{t+1}^2 }{24 L^2 ( G^2 + \sum_{i=1}^{t+1} \bar{G}_i^2)} \nonumber\\
 & \overset{(c)}{\leq} \frac{c^2 \bar{\kappa}^3 }{24L^2}  \ln\bigg( 1 + \sum_{t=1}^{T+1} \frac{\bar{G}_t^2}{G^2} \bigg) \nonumber\\ 
 & \overset{(d)}{\leq} \frac{c^2 \bar{\kappa}^3 }{24L^2} \ln (T+2).
  \label{Eq: AD_Coefficient_Variance_Sum}
\end{align}
where inequality $(a)$ uses the fact that $w_t \geq 2 G^2$ (see Theorem \ref{Thm: AD_Convergence_Main}), $(b)$ follows from the fact that $\bar{G}_t^2 \leq G^2$ (please see \eqref{Eq: DA_Bar_G_t}) for all $t \in \mathbb{N}$ and $(c)$ follows from Lemma \ref{Lem: AD_Sum_1overT}. Finally, $(d)$ again follows from the fact that $\bar{G}_t^2 \leq G^2$ for all $t \in \mathbb{N}$.

Replacing \eqref{Eq: AD_Coefficient_Variance_Sum} in \eqref{eq: AD_Half_Potential_Sum}, we get:
\begin{align}
   \frac{K}{48L^2} \sum_{t=1}^T \mathbb{E} \bigg[ \frac{\|\bar{e}_{t+1} \|^2}{\eta_t}-  \frac{\|\bar{e}_{t} \|^2}{\eta_{t-1}} \bigg] & \leq  \sum_{t=1}^T \mathbb{E} \bigg[ -\frac{\eta_t}{2}  \|  \bar{e}_{t} \|^2  +    \frac{\eta_{t} }{12} \mathbb{E}\|  \nabla f(\bar{x}_t) \|^2 \bigg] + \frac{c^2 \bar{\kappa}^3}{24L^2} \ln (T+2)  .
   \label{eq: AD_Half_Potential_Final}
\end{align}
Adding \eqref{eq: AD_Smoothness_Sum} and \eqref{eq: AD_Half_Potential_Final} above and using the definition of the potential function $\displaystyle \Phi_t = f(\bar{x}_t) + \frac{K}{48L^2 \eta_{t-1}} \|\bar{e}_t\|^2$ given in \eqref{Eq: AD_PotentialFn} we get:
\begin{align*}
    \mathbb{E}[\Phi_{T+1} - \Phi_1] & \leq  \sum_{t=1}^T \mathbb{E} \bigg[ -\frac{\eta_{t} }{2} \|  \nabla f(\bar{x}_t) \|^2 +   \frac{\eta_{t} }{12}  \|  \nabla f(\bar{x}_t) \|^2 \bigg] + \frac{c^2 \bar{\kappa}^3 }{24L^2} \ln (T+2) \\
    & =  \sum_{t=1}^T \mathbb{E} \bigg[ -  \frac{5 \eta_{t} }{12} \|  \nabla f(\bar{x}_t) \|^2 \bigg]  + \frac{c^2 \bar{\kappa}^3}{24L^2} \ln (T+2).
\end{align*}
Rearranging the terms we get;
\begin{align*}
 \mathbb{E} \bigg[  \sum_{t=1}^T  \eta_{t}\|  \nabla f(\bar{x}_t) \|^2 \bigg]    & \leq \frac{12}{5} \mathbb{E}[\Phi_{1} - \Phi_{T+1}] + \frac{c^2 \bar{\kappa}^3}{10 L^2} \ln (T+2) \\
  & \overset{(a)}{\leq} 3 \mathbb{E}[f(\bar{x}_1) - f(x^\ast)] + \frac{K}{20 L^2 \eta_0} \mathbb{E}\| \bar{e}_1\|^2+ \frac{c^2 \bar{\kappa}^3 }{10 L^2} \ln (T+2) \\
  & \overset{(b)}{\leq} 3 \mathbb{E}[f(\bar{x}_1) - f(x^\ast)] + \frac{w_0^{1/3}\sigma^2 }{20 L^2 \bar{\kappa}}  + \frac{c^2 \bar{\kappa}^3 }{10 L^2} \ln (T+2).
\end{align*}
where $(a)$ follows from the definition of $\Phi_1$ and the fact that $\Phi_{T+1} \geq f(x^\ast)$ and inequality $(b)$ uses the definition of $\eta_0$ and Lemma \ref{Lem: AD_e_bar_bound} to bound $\| \bar{e}_1 \|^2$.

Finally, note from the choice of $\displaystyle w_t = \max \bigg\{2 G^2, \bar{\kappa}^3L^3 - \sum_{i=1}^t \bar{G}_i^2,  \frac{\bar{\kappa}^3 c^3}{L^3} \bigg\}$ in the statement $(iv)$ of Theorem \ref{Thm: Convergence_Main} and the definition of $\displaystyle \eta_t = \frac{\bar{\kappa}}{(w_t + \sum_{i=1}^t \bar{G}_i^2 )^{1/3}}$ that $\eta_t$ is non-increasing with $t$. Therefore, using the fact that $\eta_T \leq \eta_t$ for all $t \in [T]$ in the above, we get
\begin{align*}
 \mathbb{E} \bigg[ \eta_T    \sum_{t=1}^T  \|  \nabla f(\bar{x}_{t}) \|^2  \bigg]   & \leq  3    \big(f(\bar{x}_1) - f(x^\ast)\big) +   \frac{w_0^{1/3}\sigma^2 }{20 L^2 \bar{\kappa}}  +  \frac{c^2 \bar{\kappa}^3 }{10 L^2} \ln (T+2). 
\end{align*}
Now let us denote:
\begin{align*}
  M(\bar{\kappa}, w_0, c, \sigma) = 3    \big(f(\bar{x}_1) - f(x^\ast)\big) +   \frac{w_0^{1/3}\sigma^2 }{20 L^2 \bar{\kappa}}  +  \frac{c^2 \bar{\kappa}^3}{10 L^2} \ln (T+2). 
\end{align*}
Using the analysis similar to one conducted in \cite{Cutkosky_NIPS2019}, we get:
\begin{align}
 \mathbb{E} \Bigg[  \sqrt{  \sum_{t=1}^T  \|  \nabla f(\bar{x}_{t}) \|^2 }  \Bigg]^2   & \leq
 \mathbb{E} \bigg[\frac{M(\bar{\kappa}, w_0,c,\sigma)}{\eta_T} \bigg] \nonumber\\
 & \overset{(a)}{=}  \mathbb{E} \bigg[\frac{M(\bar{\kappa}, w_0,c,\sigma) (w_T + \sum_{t=1}^T \bar{G}_t^2)^{1/3} }{\bar{\kappa}} \bigg]  \nonumber\\
  & \overset{(b)}{\leq}  \mathbb{E} \bigg[\frac{M(\bar{\kappa}, w_0,c,\sigma) (w_T +  G^2 T)^{1/3} }{\bar{\kappa}} \bigg] 
   \label{Eq: DA_Descent_Gradient1}
 \end{align}
 where $(a)$ follows from the definition of $\eta_T$ and $(b)$ uses the fact that $\bar{G}_t \leq G^2$ which follows from:
  \begin{align}
   \bar{G}_t^2  & = \frac{1}{K} \sum_{k=1}^K \big( g_t^{(k)} \big)^2  = \frac{1}{K} \sum_{k=1}^K \| \nabla f^{(k)}(x^{(k)}; \xi_t^{(k)})\|^2 \leq G^2.
  \label{Eq: DA_Bar_G_t}
\end{align}

Hence, we have the proof. 
\end{proof}

Using Theorem \ref{Thm: AD_Convergence_Main} and utilizing the definition of $w_T$, we can now compute the computation complexity (Definition \ref{Def: ComputationComplexity}) of the algorithm. 
\begin{cor}
\label{Cor: DA_Computation_Complexity}
 For $\displaystyle \alpha = \frac{2}{3}$ and rest of the parameters chosen according to Theorem \ref{Thm: AD_Convergence_Main}.
 \begin{enumerate}[(i)]
    \item  For $K^{1 - \alpha} \geq 56b/2^{1/3}$, we have:
\begin{align*}
      \mathbb{E}\|  \nabla f({x}_{a}) \| & \leq  O \bigg(\frac{1 + \sigma +  \sqrt{\ln(T+2)}}{K^{1/3} T^{1/2}}\bigg) + O \bigg(\frac{1 + \sigma  +  \sqrt{\ln(T+2)}}{K^{1/3} T^{1/3}}\bigg).
\end{align*}
and for $K^{1 - \alpha} \leq 56b/2^{1/3}$ we have
\begin{align*}
       \mathbb{E}\|  \nabla f({x}_{a}) \| & \leq  O \bigg(\frac{1 + \sigma +  \sqrt{\ln(T+2)}}{K^{1/2} T^{1/2}}\bigg) + O \bigg(\frac{1 + \sigma  +  \sqrt{\ln(T+2)}}{K^{1/3} T^{1/3}}\bigg).
\end{align*}
\item To reach an $\epsilon$-stationary solution we need $\tilde{O}(K^{-1} \epsilon^{-3})$, gradient computations at each node, thereby, achieving linear speedup with the number of WNs $K$ in the network.
\end{enumerate}
\end{cor}
\begin{proof}
We know from the statement of Theorem \ref{Thm: AD_Convergence_Main} that $\displaystyle w_t \leq  G^2 \max\bigg\{2 ,~ b^3 K^{3 \alpha} - \sum_{i=1}^t \frac{\bar{G}_i^2}{G^2},  ~\frac{(56b)^3}{K^{3 - 3\alpha}} \bigg\}$. This implies that, for $t = 0$, in the worst case we will have $w_0 = O(b^3 K^{3\alpha})$. Note that here the worst case refers to the worst case speedup achievable in terms of the number of WNs, $K$, present in the network. 

Moreover, after a finite number of iterations, specifically, for any $T$ such that $\displaystyle \mathbb{P}\Big[\sum_{i=1}^T ({\bar{G}_i^2}/{G^2})  \geq b^3 K^{3\alpha} \Big] = 1$, we will have $\displaystyle w_T \leq G^2 \max\bigg\{2 , ~\frac{(56b)^3}{K^{3 - 3\alpha}} \bigg\}$, i.e., a constant. Note that this follows because even if the gradients $\nabla f(\bar{x}_t)$ go to zero, the variance of the stochastic gradients of the individual nodes (variance of $\nabla f^{(k)}(x_t^{(k)}; \xi_t^{(k)})$) keeps them from going to zero. Now we consider two regimes as: \vspace{0.1 in}\\
{\bf Regime 1:} When we have $K^{1 - \alpha} \geq 56b/2^{1/3}$. \vspace{0.05 in}\\
This means $\displaystyle 2 \geq {(56b)^3}/{K^{3 - 3\alpha}}$ which further implies that we have $w_T = O(2 G^2)$. So under Regime 1, using $w_T = 2 G^2$ and $w_0 = b^3 K^{3\alpha}$ along with $\displaystyle \bar{\kappa} = \frac{b K^\alpha G^{2/3}}{L}$ and $\displaystyle c \leq \frac{56L^2}{K}$ as given in the statement of Theorem \ref{Thm: AD_Convergence_Main} in the result of Theorem \ref{Thm: AD_Convergence_Main} we get:
    \begin{align*}
   \frac{1}{T} \sum_{t=1}^T \mathbb{E}  \|  \nabla f(\bar{x}_{t}) \| & \leq  O \bigg(\frac{1}{K^{\alpha/2} T^{1/2}} + \frac{\sigma}{K^{\alpha/2} T^{1/2}} + \frac{\sqrt{\ln(T+2)}}{K^{1 - \alpha} T^{1/2}} \bigg) +  O\bigg(\frac{1}{K^{\alpha/2} T^{1/3}} + \frac{\sigma }{K^{\alpha/2} T^{1/3}} + \frac{ \sqrt{\ln(T+2)}}{K^{1 - \alpha} T^{1/3}} \bigg). 
\end{align*}
This follows by using $(x + y)^p \leq x^p + y^p$ for $x , y \geq 0$ and $p \leq 1$, to expand the terms in $(w_T + G^2T)^{1/6}$ and $(M(\bar{\kappa}, w_0, c,\sigma))^{1/2}$. Specifically, treating terms $3\big(f(\bar{x}_1) - f(x^\ast)\big)$, $\displaystyle \frac{w_0^{1/3}\sigma^2}{20 L^2 \bar{\kappa}}$  and  $\displaystyle \frac{c^2 \bar{\kappa}^3}{10 L^2} \ln (T+2)$ as three separate terms for expanding the powers of $M(\bar{\kappa}, w_0, c,\sigma)$. Moreover, on the left hand side of the inequality we have used Cauchy-Schwartz inequality to get:
\begin{align*}
       \frac{1}{T} \sum_{t=1}^T   \|  \nabla f(\bar{x}_{t}) \| \leq    \sqrt{ \frac{1}{T} \sum_{t=1}^T  \|  \nabla f(\bar{x}_{t}) \|^2 }  .
\end{align*}
Choosing $\displaystyle \alpha = \frac{2}{3}$ we get:
\begin{align*}
     \frac{1}{T} \sum_{t=1}^T   \mathbb{E}\|  \nabla f(\bar{x}_{t}) \| & \leq  O \bigg(\frac{1 + \sigma +  \sqrt{\ln(T+2)}}{K^{1/3} T^{1/2}}\bigg) + O \bigg(\frac{1 + \sigma  +  \sqrt{\ln(T+2)}}{K^{1/3} T^{1/3}}\bigg).
\end{align*}
Therefore, we have $(i)$ under Regime 1. Moreover, to obtain the $\epsilon$-stationary solution we need:
\begin{align*}
 O \bigg(\frac{\ln(T+2)}{K^{1/3} T^{1/3}} \bigg) \leq \epsilon \qquad \Rightarrow \qquad T \geq \tilde{O}(K^{-1} \epsilon^{-3}).
\end{align*}
Now, we consider Regime 2 as:\vspace{0.1 in}\\
{\bf Regime 2:} When we have $K^{1 - \alpha} \leq 56b/2^{1/3}$. \vspace{0.05 in}\\
This means $\displaystyle 2 \leq {(56b)^3}/{K^{3 - 3\alpha}}$ which further implies that we have $w_T = O(1/ K^{3-3\alpha})$.  

So under Regime 2, using $w_T =  1/ K^{3-3\alpha}$ along with $\displaystyle \bar{\kappa} = \frac{b K^\alpha G^{2/3}}{L}$ and $\displaystyle c \leq \frac{56L^2}{K}$ as given in the statement of Theorem \ref{Thm: AD_Convergence_Main} in the result of Theorem \ref{Thm: AD_Convergence_Main} we get:
    \begin{align*}
   \frac{1}{T} \sum_{t=1}^T \mathbb{E}  \|  \nabla f(\bar{x}_{t}) \| & \leq  O \bigg(\frac{1}{K^{1/2} T^{1/2}} + \frac{\sigma}{K^{1/2} T^{1/2}} + \frac{\sqrt{\ln(T+2)}}{K^{(3-3\alpha)/2} T^{1/2}} \bigg)  +  O\bigg(\frac{1}{K^{\alpha/2} T^{1/3}} + \frac{\sigma}{K^{\alpha/2} T^{1/3}} + \frac{\sqrt{\ln(T+2)}}{K^{1 - \alpha} T^{1/3}} \bigg).
\end{align*}
Again, choosing $\displaystyle \alpha = \frac{2}{3}$ we get:
\begin{align*}
     \frac{1}{T} \sum_{t=1}^T   \mathbb{E}\|  \nabla f(\bar{x}_{t}) \| & \leq  O \bigg(\frac{1 + \sigma +  \sqrt{\ln(T+2)}}{K^{1/2} T^{1/2}}\bigg) + O \bigg(\frac{1 + \sigma +   \sqrt{\ln(T+2)}}{K^{1/3} T^{1/3}}\bigg).
\end{align*}
Therefore, we have $(i)$ under Regime 2. Moreover, to achieve $\epsilon$-stationary solution we need:
\begin{align*}
 O \bigg(\frac{\ln(T+2)}{K^{1/3} T^{1/3}} \bigg) \leq \epsilon \qquad \Rightarrow \qquad T \geq \tilde{O}(K^{-1} \epsilon^{-3}).
\end{align*}
Therefore, we have the corollary.
\end{proof}
\begin{rem}
Centralized STORM proposed in \cite{Cutkosky_NIPS2019}, requires $\tilde{O}(\epsilon^{-3})$ gradient computations to achieve an $\epsilon$-stationary solution. In contrast to the centralize STORM, Corollary \ref{Cor: DA_Computation_Complexity} given above implies that for AD-STORM the total number of gradient evaluations at each WN in the worst case is reduced by a factor of $K$. This implies that AD-STORM is capable of achieving linear speedup with the number of WNs, $K$, while at the same time achieving optimal computational complexity compared to the state-of-the-art up to logarithmic factors \cite{Arjevani_Carmon_2019_LowerBounds}. 
\end{rem}
\begin{rem}
Note that the design of AD-STORM given in Algorithm \ref{Algo_AD-STORM} requires knowledge of the gradient bound, $G$, given in Assumption \ref{Ass: Unbiased_Var_Grad}. Moreover, the design of AD-STORM does not rely on the knowledge of the variance parameter $\sigma^2$, however, the convergence depends $\sigma^2$.
\end{rem}
Next, we present a non-adaptive version of the algorithm, D-STORM. The proposed non-adaptive algorithm is a special case of the adaptive algorithm which does not require the knowledge of $G$. Moreover, it does not even require the bounded gradient Assumption \ref{Ass: Unbiased_Var_Grad} to hold true. However, the design of the non-adaptive algorithm requires the knowledge of the variance parameter $\sigma^2$ to design the step sizes.
\section{Non-Adaptive Distributed Algorithm: D-STORM}
\label{sec: D-STORM}
\begin{algorithm}[t]
\caption{Distributed STORM - D-STORM}
\label{Algo_D-STORM}
\begin{algorithmic}[1]
	\State{\textbf{Input}: Parameters: $\bar{\kappa}$, $\{w_t\}_{t=0}^{T}$ and $c$}
	\State{\textbf{Initialize}:  Iterate $x_1^{(k)} = \bar{x}_1$, descent direction $d_1^{(k)} = \bar{d}_1 = \frac{1}{K}\sum_{k=1}^K \nabla f^{(k)}(x_1^{(k)} ; \xi_1^{(k)})$ for all $k \in [K]$, step size $\eta_0 = \frac{\bar{\kappa}}{w_0^{1/3}}$.}
	\For{$t = 1$ to $T$}
    	\For{$k = 1$ to $K$}
    	\State{$ \eta_{t} = \frac{\bar{\kappa}}{(w_t + \sigma^2 t )^{1/3}}$}
        	\State{$ x_{t+1}^{(k)} =  x_t^{(k)} - \eta_{t} d_{t}^{(k)}$}
        			\State{$a_{t+1} = c \eta_{t}^2$}
        	\State{$ d_{t+1}^{(k)} = \nabla f^{(k)}(x_{t+1}^{(k)} ; \xi_{t+1}^{(k)}) + (1 - a_{t+1})      \big(  d_{t}^{(k)} - \nabla f^{(k)}(x_{t}^{(k)} ; \xi_{t+1}^{(k)}) \big)$ \hfill $\to$ Forward $d_{t+1}^{(k)}$ to SN}
        	\State{$ d_{t+1}^{(k)} = \bar{d}_{t+1} = \frac{1}{K} \sum_{k=1}^K d_{t+1}^{(k)}$ \hfill $\to$ Receive $\bar{d}_{t+1}$ from SN}
	    \EndFor
	\EndFor
\State{{\bf Return:} $x_a$ chosen uniformly randomly from $\{x_t\}_{t=1}^T$}	
\end{algorithmic}
\end{algorithm}
In this section, we present the non-adaptive version of the distributed algorithm developed in Section \ref{sec: AD-STORM}. As pointed out earlier, the proposed algorithm, D-STORM, does not rely on the Bounded Gradient Assumption given in Assumption \ref{Ass: Unbiased_Var_Grad}. In fact D-STORM, replaces $G^2$ by $\sigma^2$ in the design of the step sizes and still guarantees the same convergence as for AD-STORM.  

The steps of the algorithm D-STORM are given in Algorithm \ref{Algo_D-STORM}. After specifying a few parameters (please see Theorem \ref{Thm: Convergence_Main}), in Step 2 of the algorithm we initialize the algorithm with the same initial iterate, $x_1^{(k)} = \bar{x}_1$, at each WN. Also, each node uses the same initial descent direction, $d_1^{(k)} = \bar{d}_1$, which is constructed using unbiased estimates of the gradients at individual WNs, $\nabla f^{(k)}(x_1^{(k)}; \xi_1^{(k)})$, for all $k \in [K]$ with $\xi_1^{(k)} \sim \mathcal{D}^{(k)}$. Each node then computes the step size, $\eta_t$, according to Step 5 and then updates the iterate in Step 6 of the algorithm. Note at this stage that as each node had the same initial iterate, $\bar{x}_1$, and each node uses the same descent direction, the updated iterate, $x_{t+1}^{k}$, will be same across all $k \in [K]$. Therefore, we denote  $x_{t+1}^{k} = \bar{x}_{t+1}$ for $t \in [T]$. Then in Step 7, the momentum parameter is updated which is then used to compute the new local descent direction, $d_{t+1}^{(k)}$, in Step 8. Finally, in Step 9 of the algorithm, the local descent directions, $d_{t+1}^{(k)}$, are forwarded to the SN and updated descent direction, $\bar{d}_{t+1}$, are received from the SN at the WNs. The process is repeated until convergence. 

Next, we present the convergence guarantees associated with the algorithm. 
\subsection{Analysis: D-STORM}
\label{subsec: D-STORM_Convergence}
The proof for the convergence of D-STORM follows the same approach as for the AD-STORM. However, the proof is relatively simpler as the step size $\eta_t$ for D-STORM does not depend on the stochastic gradients and is thereby deterministic.

First, we present the Descent lemma. 
\begin{lemma}[Descent Lemma]
\label{Lem: DescentLemma}
For $\eta_t \leq \frac{1}{L}$ and $\bar{e}_t = \bar{d}_t - \nabla f(\bar{x}_{t})$, we have:
\begin{align*}
  \mathbb{E} f(\bar{x}_{t + 1})   \leq   \mathbb{E}  f(\bar{x}_{t })  - \frac{\eta_t}{2}   \mathbb{E} \|\nabla f(\bar{x}_t) \|^2 + \frac{\eta_t}{2}   \mathbb{E} \| \bar{e}_t  \|^2. 
\end{align*}
\end{lemma}
The proof follows exactly the same approach as the proof of Lemma \ref{Lem: AD_DescentLemma}.
Next, we present the lemma for error contraction. 
\begin{lemma}[Error Contraction] 
\label{Lem: ErrorContractionLemma}
The error term $\bar{e}_t$ from Algorithm \ref{Algo_D-STORM} satisfies the following:
\begin{align*}
    \mathbb{E} \| \bar{e}_t \|^2  \leq  \bigg( (1 - a_t)^2  + \frac{4 (1 - a_t)^2 L^2 \eta_{t-1}^2}{K} \bigg) \mathbb{E}\|  \bar{e}_{t-1} \|^2 + \frac{4 (1 - a_t)^2 L^2 \eta_{t-1}^2}{K}  \mathbb{E}\|  \nabla f(\bar{x}_{t-1}) \|^2 + \frac{2a_t^2 \sigma^2}{K}.
\end{align*}
\end{lemma}
\begin{proof}
Using the definition of $\bar{e}_t$ we have
\begin{align}
   & \mathbb{E} \| \bar{e}_t \|^2  =  \mathbb{E} \big\| \bar{d}_t - \nabla f(\bar{x}_t)  \big\|^2 \nonumber \\
    & \overset{(a)}{=} \mathbb{E} \bigg\| \frac{1}{K} \sum_{k = 1}^K \nabla f^{(k)}(x^{(k)}_{t};\xi_t^{(k)}) + (1 - a_t)\left( \bar{d}_{t-1} - \frac{1}{K} \sum_{k = 1}^K \nabla f^{(k)}(x^{(k)}_{t - 1}; \xi_t^{(k)})\right) - \nabla f(\bar{x}_t)   \bigg\|^2 \nonumber\\
    &  \overset{(b)}{=} \mathbb{E} \bigg\| \frac{1}{K} \sum_{k = 1}^K \left[\left( \nabla f^{(k)}(\bar{x}_t;\xi_t^{(k)})  -  \nabla f^{(k)}(\bar{x}_t) \right)    - (1 - a_t) \left( \nabla f^{(k)}(\bar{x}_{t - 1}; \xi_t^{(k)}) - \nabla f^{(k)}(\bar{x}_{t - 1})\right) \right] + (1 - a_t) \bar{e}_{t-1}   \bigg\|^2 \nonumber \\
    &   \overset{(c)}{=} (1 - a_t)^2 \mathbb{E}\| \bar{e}_{t-1}\|^2  + \frac{1}{K^2 }\mathbb{E}\bigg\|\sum_{k = 1}^K \left[\left( \nabla f^{(k)}(\bar{x}_t;\xi_t^{(k)})  -  \nabla f^{(k)}(\bar{x}_t) \right)    - (1 - a_t) \left( \nabla f^{(k)}(\bar{x}_{t - 1}; \xi_t^{(k)}) - \nabla f^{(k)}(\bar{x}_{t - 1})\right) \right] \bigg\|^2 \nonumber\\
    & \qquad +2 \underbrace{\mathbb{E} \left\langle  (1 - a_t) \bar{e}_{t-1}, \frac{1}{K}\sum_{k = 1}^K \left[\left( \nabla f^{(k)}(\bar{x}_t;\xi_t^{(k)})  -  \nabla f^{(k)}(\bar{x}_t) \right)    - (1 - a_t) \left( \nabla f^{(k)}(\bar{x}_{t - 1}; \xi_t^{(k)}) - \nabla f^{(k)}(\bar{x}_{t - 1})\right) \right]   \right\rangle}_{=0} \nonumber \\
    &  \overset{(d)}{=} (1 - a_t)^2 \mathbb{E}\| \bar{e}_{t-1}\|^2 + \frac{1}{K^2 } \sum_{k = 1}^K  \mathbb{E}\big\| \left( \nabla f^{(k)}(\bar{x}_t;\xi_t^{(k)})  -  \nabla f^{(k)}(\bar{x}_t) \right)    - (1 - a_t) \left( \nabla f^{(k)}(\bar{x}_{t - 1}; \xi_t^{(k)}) - \nabla f^{(k)}(\bar{x}_{t - 1})\right)  \big\|^2 \nonumber\\
    & \qquad \qquad + \frac{1}{K^2} \sum_{k,l \in [K],k \neq l} \mathbb{E} \bigg\langle   \left( \nabla f^{(k)}(\bar{x}_t;\xi_t^{(k)})  -  \nabla f^{(k)}(\bar{x}_t) \right)    - (1 - a_t) \left( \nabla f^{(k)}(\bar{x}_{t - 1}; \xi_t^{(k)}) - \nabla f^{(k)}(\bar{x}_{t - 1})\right), \nonumber\\
   & \qquad \qquad \qquad  \qquad  \quad    \underbrace{ \qquad \quad \left( \nabla f^{(l)}(\bar{x}_t;\xi_t^{(l)})  -  \nabla f^{(l)}(\bar{x}_t) \right)    - (1 - a_t) \left( \nabla f^{(l)}(\bar{x}_{t - 1}; \xi_t^{(l)}) - \nabla f^{(l)}(\bar{x}_{t - 1})\right) \bigg\rangle}_{=0} \nonumber\\
   & \overset{(e)}{=} (1 - a_t)^2 \mathbb{E}\| \bar{e}_{t-1}\|^2 + \frac{1}{K^2 } \sum_{k = 1}^K  \mathbb{E}\big\| \left( \nabla f^{(k)}(\bar{x}_{t};\xi_t^{(k)})  -  \nabla f^{(k)}(\bar{x}_t) \right)    - (1 - a_t) \left( \nabla f^{(k)}(\bar{x}_{t - 1}; \xi_t^{(k)}) - \nabla f^{(k)}(\bar{x}_{t - 1})\right)  \big\|^2.
   \label{Eq: Error_Contraction1}
\end{align}
where $(a)$ follows from the definition of the descent direction $d_t^{(k)}$ given in Step 9 of Algorithm \ref{Algo_D-STORM}, $(b)$ follows from adding and subtracting $(1- a_t) \frac{1}{K} \sum_{k=1}^K \nabla f^{(k)}(\bar{x}_{t-1}) (= (1 - a_t) \nabla f(\bar{x}_{t-1})$), the fact that $x_t^{(k)} = \bar{x}_t$ for all $t \in [T]$ (please see Step 6 of Algorithm \ref{Algo_D-STORM}) and using the definition of $\bar{e}_{t-1}$, $(c)$ follows from expanding the norm using inner products, $(d)$ follows from Lemma \ref{Lem: InnerProduct_e_t_Grad} and again expanding the norm using the inner products, finally $(e)$ results from the usage of Lemma \ref{Lem: InnerProd_AcrossNodes}.

Now let us consider the 2nd term of \eqref{Eq: Error_Contraction1} above we have:
\begin{align}
  &  \frac{1}{K^2} \sum_{k = 1}^K  \mathbb{E}\big\| \left( \nabla f^{(k)}(\bar{x}_{t};\xi_t^{(k)})  -  \nabla f^{(k)}(\bar{x}_{t}) \right)    - (1 - a_t) \left( \nabla f^{(k)}(\bar{x}_{t - 1}; \xi_t^{(k)}) - \nabla f^{(k)}(\bar{x}_{t - 1})\right)  \big\|^2 \nonumber\\
    & =  \frac{1}{K^2} \sum_{k = 1}^K  \mathbb{E}\big\| (1 -a_t) \left[ \left( \nabla f^{(k)}(\bar{x}_{t};\xi_t^{(k)})  -  \nabla f^{(k)}(\bar{x}_{t}) \right)    - \left( \nabla f^{(k)}(\bar{x}_{t - 1}; \xi_t^{(k)}) - \nabla f^{(k)}(\bar{x}_{t - 1})\right) \right] \nonumber \\
    & \qquad \qquad \qquad  \qquad \qquad \qquad \qquad  \qquad \qquad \qquad \qquad    + a_t   \left( \nabla f^{(k)}(\bar{x}_{t}; \xi_t^{(k)}) - \nabla f^{(k)}(\bar{x}_{t})\right)      \big\|^2 \nonumber\\
    & \overset{(a)}{\leq} \frac{2 (1 - a_t)^2}{K^2} \sum_{k=1}^K \mathbb{E} \big\| \left( \nabla f^{(k)}(\bar{x}_{t};\xi_t^{(k)})  -  \nabla f^{(k)}(\bar{x}_{t - 1}; \xi_t^{(k)}) \right)    -  \left(\nabla f^{(k)}(\bar{x}_{t}) - \nabla f^{(k)}(\bar{x}_{t - 1})\right) \big\|^2 \nonumber\\
    & \qquad \qquad \qquad  \qquad \qquad \qquad \qquad \qquad \qquad \qquad   + \frac{2 a_t^2}{K^2} \sum_{k=1}^K \mathbb{E}\big\|    \nabla f^{(k)}(\bar{x}_{t}; \xi_t^{(k)}) - \nabla f^{(k)}(\bar{x}_{t})    \big\|^2 \nonumber\\
    & \overset{(b)}{\leq} \frac{2 (1 - a_t)^2}{K^2} \sum_{k=1}^K \mathbb{E} \big\|   \nabla f^{(k)}(\bar{x}_t;\xi_t^{(k)})  -  \nabla f^{(k)}(\bar{x}_{t - 1}; \xi_t^{(k)}) \big\|^2 + \frac{2a_t^2 \sigma^2}{K} \nonumber\\
    & \overset{(c)}{\leq} \frac{2 (1 - a_t)^2 L^2}{K^2} \sum_{k=1}^K \mathbb{E}\| \bar{x}_t - \bar{x}_{t-1} \|^2 + \frac{2a_t^2 \sigma^2}{K} \nonumber \\
    & \overset{(d)}{=} \frac{2 (1 - a_t)^2 L^2 \eta_{t-1}^2}{K}  \mathbb{E}\|  \bar{d}_{t-1} \|^2 + \frac{2a_t^2 \sigma^2}{K}.
    \label{Eq: Error_Contraction2}
\end{align}
where $(a)$ follows from Lemma \ref{Lem: Norm_Ineq}, $(b)$ follows from Lemma \ref{Lem: Mean_Variance_Bound} and the variance bound given in Assumption \ref{Ass: Unbiased_Var_Grad}, $(c)$ follows from the Lipschitz continuity of the gradient given in Assumption \ref{Ass: Lip_Smoothness} and, finally, $(d)$ follows from the update Step 6 of Algorithm \ref{Algo_D-STORM}.

Replacing, \eqref{Eq: Error_Contraction2} in \eqref{Eq: Error_Contraction1} we get:
\begin{align}
    \mathbb{E} \| \bar{e}_t \|^2  & \leq (1 - a_t)^2 \mathbb{E}\| \bar{e}_{t-1}\|^2 + \frac{2 (1 - a_t)^2 L^2 \eta_{t-1}^2}{K}  \mathbb{E}\|  \bar{d}_{t-1} \|^2 + \frac{2a_t^2 \sigma^2}{K} \nonumber \\
     & \overset{(a)}{\leq} \bigg( (1 - a_t)^2  + \frac{4 (1 - a_t)^2 L^2 \eta_{t-1}^2}{K} \bigg) \mathbb{E}\|  \bar{e}_{t-1} \|^2 + \frac{4 (1 - a_t)^2 L^2 \eta_{t-1}^2}{K}  \mathbb{E}\|  \nabla f(\bar{x}_{t-1}) \|^2 + \frac{2a_t^2 \sigma^2}{K}.
     \label{eq: D_ErrorDescent}
\end{align}
where $(a)$ above follows by adding and subtracting $\nabla f(\bar{x}_{t-1})$ inside the norm $\| \bar{d}_{t-1}\|^2$ and using Lemma \ref{Lem: Norm_Ineq}. 

Therefore, we have the result.
\end{proof}
Note the similarity of \eqref{eq: D_ErrorDescent} with \eqref{eq: AD_ErrorDescent}, $\bar{G}_t$ in \eqref{eq: AD_ErrorDescent} is replaced by $\sigma^2$ in \eqref{eq: D_ErrorDescent}. Moreover, since the step sizes and $\bar{G}_t$ are random in \eqref{eq: AD_ErrorDescent}, we have the expectations with all the random quantities. We again use the same Potential function as defined earlier in \eqref{Eq: AD_PotentialFn}. Here, we define it again for convenience. 
\begin{align}
\label{Eq: PotentialFn}
\Phi_t = f(\bar{x}_t) + \frac{K}{48L^2 \eta_{t-1}} \| \bar{e}_t \|^2.
\end{align}
Using the above two lemmas finally we can state the main convergence result for D-STORM in the next theorem. 
\begin{theorem}
\label{Thm: Convergence_Main}
Under the assumptions given in Section \ref{sec: Problem} and for the choice of parameters:
\begin{enumerate}[(i)]
    \item For any $\displaystyle b^3 \geq \frac{2^{2/3}}{84}$.
    \item $\displaystyle \bar{\kappa} = \frac{b K^\alpha \sigma^{2/3}}{L}$.
    \item $\displaystyle c = \frac{28L^2}{K} + \frac{2^{2/3} \sigma^2}{3 L \bar{\kappa}^3} = L^2 \bigg( \frac{28}{K} + \frac{2^{2/3}}{3 b^3 K^{3\alpha}} \bigg) \overset{(i)}{\leq} \frac{56 L^2}{K}$.
    \item $\displaystyle w_t = \max \bigg\{2 \sigma^2, \bar{\kappa}^3L^3 - \sigma^2t,  \frac{\bar{\kappa}^3 c^3}{L^3} \bigg\} \overset{(ii), (iii)}{\leq}  \sigma^2 \max\bigg\{2 ,~ b^3 K^{3 \alpha} - t,  ~\frac{(56b)^3}{K^{3 - 3\alpha}} \bigg\}$.
\end{enumerate}
Then the algorithm D-STORM satisfies:
\begin{align*}
    \mathbb{E}[\| \nabla f(x_a) \|^2] & \leq  \frac{1}{T} \bigg[ \frac{3 w_T^{1/3} \mathbb{E}[f(\bar{x}_1) - f(x^\ast)]}{\bar{\kappa}} + \frac{w_T^{1/3} w_0^{1/3} \sigma^2}{20L^2\bar{\kappa}^2} + \frac{w_T^{1/3} c^2 \bar{\kappa}^2}{10L^2} \ln(T+1) \bigg]     \\
    & \qquad \qquad \qquad \qquad + \frac{1}{T^{2/3}} \bigg[ \frac{3  \sigma^{2/3}  \mathbb{E}[f(\bar{x}_1) - f(x^\ast)]}{\bar{\kappa}} +   \frac{\sigma^{8/3}  w_0^{1/3}}{20 L^2 \bar{\kappa}^2 }  +  \frac{\sigma^{2/3}  c^2 \bar{\kappa}^2 }{10 L^2 } \ln (T+1) \bigg].
\end{align*}
\end{theorem}

\begin{proof}
First, using Lemma \ref{Lem: DescentLemma} and adding over $t = 1$ to $T$, we get:
\begin{align}
 \mathbb{E} [f(\bar{x}_{T + 1}) -   f(\bar{x}_{1})]  \leq     - \sum_{t=1}^T \frac{\eta_t}{2} \mathbb{E} \|\nabla f(\bar{x}_t) \|^2 + \sum_{t=1}^T \frac{\eta_t}{2} \mathbb{E} \| \bar{e}_t\|^2.
    \label{eq: D_Smoothness_Sum}
\end{align}
Now using Lemma \ref{Lem: ErrorContractionLemma} we compute:
\begin{align}
    \frac{\mathbb{E}\|\bar{e}_{t+1} \|^2}{\eta_t}-  \frac{\mathbb{E}\|\bar{e}_{t} \|^2}{\eta_{t-1}} & \leq \bigg( (1 - a_{t+1})^2  + \frac{4 (1 - a_{t+1})^2 L^2 \eta_{t}^2}{K} \bigg) \frac{\mathbb{E}\|  \bar{e}_{t} \|^2}{\eta_t} \nonumber\\
    & \qquad \qquad \qquad \quad + \frac{4 (1 - a_{t+1})^2 L^2 \eta_{t}}{K}  \mathbb{E}\|  \nabla f(\bar{x}_t) \|^2  + \frac{2a_{t+1}^2 \sigma^2}{\eta_t K} -   \frac{\mathbb{E}\|\bar{e}_{t} \|^2}{\eta_{t-1}} \nonumber\\
    & =  \bigg( \eta_t^{-1} (1 - a_{t+1})^2 \bigg(1 + \frac{4 L^2 \eta_{t}^2}{K} \bigg) -  \eta_{t-1}^{-1} \bigg) \mathbb{E}\|  \bar{e}_{t} \|^2 \nonumber\\
    & \qquad \qquad \qquad \qquad \qquad + \frac{4 (1 - a_{t+1})^2 L^2 \eta_{t}}{K}  \mathbb{E}\|  \nabla f(\bar{x}_t) \|^2  + \frac{2c^2 \eta_t^3 \sigma^2}{K} \nonumber\\
    & \overset{(a)}{\leq}   \bigg( \eta_t^{-1} (1 - a_{t+1}) \bigg(1 + \frac{4 L^2 \eta_{t}^2}{K} \bigg) -  \eta_{t-1}^{-1} \bigg) \mathbb{E}\|  \bar{e}_{t} \|^2 + \frac{4 L^2 \eta_{t}}{K}  \mathbb{E}\|  \nabla f(\bar{x}_t) \|^2  + \frac{2c^2 \eta_t^3 \sigma^2}{K}.
    \label{eq: Error_Potential}
\end{align}
where $(a)$ follows from the fact that $0 < 1 - a_t < 1$

Let us consider the coefficient of the first term of \eqref{eq: Error_Potential}:
\begin{align}
     \eta_t^{-1} (1 - a_{t+1}) \bigg(1 + \frac{4 L^2 \eta_{t}^2}{K} \bigg) -  \eta_{t-1}^{-1} & = \eta_t^{-1} - \eta_{t-1}^{-1} + \eta_t^{-1} \bigg( \frac{4L^2 \eta_t^2}{K} - a_{t+1} \bigg( 1 + \frac{4L^2 \eta_t^2}{K} \bigg) \bigg) \nonumber\\
     & \overset{(a)}{\leq} \eta_t^{-1} - \eta_{t-1}^{-1} + \eta_t^{-1} \bigg( \frac{4L^2 \eta_t^2}{K} - a_{t+1}   \bigg) \nonumber\\
     & = \eta_t^{-1} - \eta_{t-1}^{-1} + \eta_t  \bigg( \frac{4L^2  }{K} - c  \bigg).
     \label{Eq: Coefficient_e_t}
\end{align}
where inequality $(a)$ utilizes the fact that ${4 L^2 \eta_t^2}/{K}  > 0$.

First, considering $\eta_t^{-1} - \eta_{t-1}^{-1}$ in \eqref{Eq: Coefficient_e_t} we have from the definition of $\eta_t$ in Algorithm \ref{Algo_D-STORM}:
\begin{align}
    \eta_t^{-1} - \eta_{t-1}^{-1} & =  \frac{ \big(w_t + \sigma^2 t \big)^{1/3} - \big(w_{t-1} + \sigma^2 (t-1)\big)^{1/3} }{\bar{\kappa}} \nonumber \\
    & \overset{(a)}{\leq}     \frac{ \big(w_t + \sigma^2 t \big)^{1/3} - \big(w_{t} + \sigma^2 (t-1)\big)^{1/3} }{\bar{\kappa}} \nonumber\\
    & \overset{(b)}{\leq} \frac{\sigma^2}{3 \bar{\kappa} \big(w_t + \sigma^2 (t - 1) \big)^{2/3}} \nonumber\\
    & = \frac{\sigma^2}{3 \bar{\kappa} \big(w_t - \sigma^2 + \sigma^2 t  \big)^{2/3}} \nonumber\\
     & \overset{(c)}{\leq} \frac{\sigma^2}{3 \bar{\kappa} \big(w_t/2  + \sigma^2 t  \big)^{2/3}} \nonumber\\
        & = \frac{2^{2/3} \sigma^2}{3 \bar{\kappa} \big(w_t  + 2 \sigma^2 t  \big)^{2/3}} \nonumber \\
   & \leq \frac{2^{2/3}\sigma^2}{3 \bar{\kappa} \big(w_t + \sigma^2 t  \big)^{2/3}} = \frac{2^{2/3}\sigma^2 \bar{\kappa}^2}{3 \bar{\kappa}^3 \big(w_t + \sigma^2 t  \big)^{2/3}}  \overset{(d)}{=} \frac{2^{2/3}\sigma^2 }{3 \bar{\kappa}^3 } \eta_t^2 \overset{(e)}{\leq} \frac{2^{2/3}\sigma^2}{3 L \bar{\kappa}^3 } \eta_t.
   \label{Eq: Coefficient_e_t_1}
\end{align}
where inequality $(a)$ uses the fact that we have $w_t \leq w_{t-1}$ which follows from the definition of $w_t$ given in the statement $(iv)$ of Theorem \ref{Thm: Convergence_Main} and $(b)$ follows from:
$$(x + y)^{1/3} - x^{1/3} \leq \frac{y}{3x^{2/3}}.$$
In inequality $(c)$, we have used the fact that $w_t \geq 2 \sigma^2$, finally in $(d)$ and $(e)$ we used the definition of $n_t$ given in Algorithm \ref{Algo_D-STORM} and the fact that $\eta_t \leq 1/L$, respectively.

Now consider the term $\displaystyle \eta_t \bigg( \frac{4L^2  }{K} - c \bigg)$ in \eqref{Eq: Coefficient_e_t}, since we have $\displaystyle c = \frac{28L^2}{K} + \frac{2^{2/3}\sigma^2}{3 L\bar{\kappa}^3}$ we get:
\begin{align}
    \eta_t \bigg( \frac{4L^2  }{K} - c \bigg) = \eta_t \bigg( - \frac{24L^2}{K} - \frac{2^{2/3}\sigma^2}{3 L\bar{\kappa}^3} \bigg)
    \label{Eq: Coefficient_e_t_2}
\end{align}
Substituting \eqref{Eq: Coefficient_e_t_1} and \eqref{Eq: Coefficient_e_t_2} in \eqref{Eq: Coefficient_e_t}, we get:
\begin{align}
    \eta_t^{-1} (1 - a_{t+1}) \bigg(1 + \frac{4 L^2 \eta_{t}^2}{K} \bigg) -  \eta_{t-1}^{-1}  \leq - \frac{24 L^2}{K} \eta_t.
    \label{Eq: Eq: Coefficient_e_t_Bound}
\end{align}
Substituting \eqref{Eq: Eq: Coefficient_e_t_Bound} in \eqref{eq: Error_Potential}, we get:
\begin{align*}
     \frac{\mathbb{E}\|\bar{e}_{t+1} \|^2}{\eta_t}-  \frac{\mathbb{E}\|\bar{e}_{t} \|^2}{\eta_{t-1}} & \leq - \frac{24 L^2 \eta_t}{K}  \mathbb{E}\|  \bar{e}_{t} \|^2  + \frac{4  L^2 \eta_{t} }{K} \mathbb{E}\|  \nabla f(\bar{x}_t) \|^2  + \frac{2c^2 \sigma^2 \eta_t^3 }{K} .
\end{align*}
Now summing over $t$ and multiplying by $\displaystyle \frac{K}{48L^2}$, we get:
\begin{align}
   \frac{K}{48L^2} \sum_{t=1}^T \bigg( \frac{\mathbb{E}\|\bar{e}_{t+1} \|^2}{\eta_t}-  \frac{\mathbb{E}\|\bar{e}_{t} \|^2}{\eta_{t-1}} \bigg) & \leq - \sum_{t=1}^T \frac{\eta_t}{2}  \mathbb{E}\|  \bar{e}_{t} \|^2  +  \sum_{t=1}^T \frac{\eta_{t} }{12} \mathbb{E}\|  \nabla f(\bar{x}_t) \|^2  + \sum_{t=1}^T \frac{c^2  \sigma^2 \eta_t^3}{24 L^2} .
   \label{eq: Half_Potential_Sum}
\end{align}
Finally, considering the last term of \eqref{eq: Half_Potential_Sum} above and using the definition of $\eta_t$ from Algorithm \ref{Algo_D-STORM} we have:
\begin{align}
  \sum_{t=1}^T \frac{c^2  \sigma^2 \eta_t^3 }{24 L^2} & = \sum_{t=1}^T \frac{c^2  \sigma^2 \bar{\kappa}^3}{24 L^2 (w_t + \sigma^2 t)} \nonumber\\
  & \overset{(a)}{\leq} \sum_{t=1}^T \frac{c^2  \sigma^2 \bar{\kappa}^3}{24 L^2 (\sigma^2 + \sigma^2 t)} \nonumber\\
  &= \sum_{t=1}^T \frac{c^2  \bar{\kappa}^3}{24 L^2 (1+  t)} \nonumber\\
  & \overset{(b)}{\leq} \frac{c^2 \bar{\kappa}^3 }{24L^2} \ln (T+1).
  \label{Eq: Coefficient_Variance_Sum}
\end{align}
where inequality $(a)$ uses the fact that $w_t \geq 2 \sigma^2 > \sigma^2$ and $(b)$ follows from Lemma \ref{Lem: AD_Sum_1overT}.

Substituting \eqref{Eq: Coefficient_Variance_Sum} in \eqref{eq: Half_Potential_Sum}, we get:
\begin{align}
   \frac{K}{48L^2} \sum_{t=1}^T \bigg( \frac{\mathbb{E}\|\bar{e}_{t+1} \|^2}{\eta_t}-  \frac{\mathbb{E}\|\bar{e}_{t} \|^2}{\eta_{t-1}} \bigg) & \leq - \sum_{t=1}^T \frac{\eta_t}{2}  \mathbb{E}\|  \bar{e}_{t} \|^2  + \sum_{t=1}^T  \frac{\eta_{t} }{12} \mathbb{E}\|  \nabla f(\bar{x}_t) \|^2  + \frac{c^2 \bar{\kappa}^3 }{24L^2} \ln (T+1) .
   \label{eq: Half_Potential_Final}
\end{align}
Adding \eqref{eq: D_Smoothness_Sum} and \eqref{eq: Half_Potential_Final} above and using the definition of potential function $\displaystyle \Phi_t = f(\bar{x}_t) + \frac{K}{48L^2 \eta_{t-1}} \|\bar{e}_t\|^2$ given in \eqref{Eq: PotentialFn}, we get:
\begin{align*}
    \mathbb{E}[\Phi_{T+1} - \Phi_1] & \leq - \sum_{t=1}^T  \frac{\eta_{t} }{2} \mathbb{E}\|  \nabla f(\bar{x}_t) \|^2 + \sum_{t=1}^T  \frac{\eta_{t} }{12} \mathbb{E}\|  \nabla f(\bar{x}_t) \|^2  + \frac{c^2 \bar{\kappa}^3 }{24L^2} \ln (T+1) \\
    & = - \sum_{t=1}^T  \frac{5 \eta_{t} }{12} \mathbb{E}\|  \nabla f(\bar{x}_t) \|^2  + \frac{c^2 \bar{\kappa}^3 }{24L^2} \ln (T+1).
\end{align*}
Rearranging the terms we get;
\begin{align*}
  \sum_{t=1}^T   \eta_{t}\mathbb{E}\|  \nabla f(\bar{x}_t) \|^2     & \leq \frac{12}{5} \mathbb{E}[\Phi_{1} - \Phi_{T+1}] + \frac{c^2 \bar{\kappa}^3 }{10 L^2} \ln (T+1) \\
  & \overset{(a)}{\leq} 3 \mathbb{E}[f(\bar{x}_1) - f(x^\ast)] + \frac{K}{20 L^2 \eta_0} \mathbb{E}\| \bar{e}_1\|^2+ \frac{c^2 \bar{\kappa}^3 }{10 L^2} \ln (T+1) \\
  & \overset{(b)}{\leq} 3 \mathbb{E}[f(\bar{x}_1) - f(x^\ast)] + \frac{w_0^{1/3}\sigma^2 }{20 L^2 \bar{\kappa}}  + \frac{c^2 \bar{\kappa}^3 }{10 L^2} \ln (T+1).
\end{align*}
where $(a)$ follows from the definition of $\Phi_1$ and the fact that $\Phi_{T+1} \geq f(x^\ast)$ and inequality $(b)$ uses the definition of $\eta_0$ and Lemma \ref{Lem: AD_e_bar_bound} to bound $\| \bar{e}_1 \|^2$.

Finally, note from the choice of $\displaystyle w_t = \max \bigg\{2 \sigma^2, \bar{\kappa}^3L^3 - \sigma^2t,  \frac{\bar{\kappa}^3 c^3}{L^3} \bigg\}$ in the statement $(iv)$ of Theorem \ref{Thm: Convergence_Main} and the definition of $\displaystyle \eta_t = \frac{\bar{\kappa}}{(w_t + \sigma^2 t )^{1/3}}$ that $\eta_t$ is non-increasing with $t$. Therefore, using the fact that $\eta_T \leq \eta_t$ for all $t \in [T]$ in above, we get
\begin{align*}
\eta_T    \sum_{t=1}^T   \mathbb{E}\|  \nabla f(\bar{x}_{t}) \|^2     & \leq  3    \mathbb{E}[f(\bar{x}_1) - f(x^\ast)] +   \frac{w_0^{1/3}\sigma^2 }{20 L^2 \bar{\kappa}}  +  \frac{c^2 \bar{\kappa}^3 }{10 L^2} \ln (T+1). 
\end{align*}
Substituting $\displaystyle \eta_T = \frac{\bar{\kappa}}{(w_T + \sigma^2 T)^{1/3}}$ in the above, we get
\begin{align*}
\frac{1}{T} \sum_{t=1}^T   \mathbb{E}\|  \nabla f(\bar{x}_{t}) \|^2  & \leq   \frac{3 (w_T + \sigma^2 T)^{1/3} \mathbb{E}[f(\bar{x}_1) - f(x^\ast)]}{\bar{\kappa} T} +   \frac{(w_T + \sigma^2 T)^{1/3} w_0^{1/3}\sigma^2 }{20 L^2 \bar{\kappa}^2 T}  +  \frac{(w_T + \sigma^2 T)^{1/3} c^2 \bar{\kappa}^2 }{10 L^2 T}  \ln (T+1).
\end{align*}
Using the identity $(x + y)^{1/3} \leq x^{1/3} + y^{1/3}$, we have:
\begin{align*}
    \frac{1}{T} \sum_{t=1}^T   \mathbb{E}\|  \nabla f(\bar{x}_{t}) \|^2  & \leq  \frac{1}{T} \bigg[ \frac{3 w_T^{1/3} \mathbb{E}[f(\bar{x}_1) - f(x^\ast)]}{\bar{\kappa}} + \frac{w_T^{1/3} w_0^{1/3} \sigma^2}{20L^2\bar{\kappa}^2} + \frac{w_T^{1/3} c^2 \bar{\kappa}^2}{10L^2} \ln(T+1) \bigg]     \\
    & \qquad \qquad \quad + \frac{1}{T^{2/3}} \bigg[ \frac{3  \sigma^{2/3}  \mathbb{E}[f(\bar{x}_1) - f(x^\ast)]}{\bar{\kappa}} +   \frac{\sigma^{8/3}  w_0^{1/3}}{20 L^2 \bar{\kappa}^2 }  +  \frac{\sigma^{2/3}  c^2 \bar{\kappa}^2 }{10 L^2 } \ln (T+1) \bigg].
\end{align*}
Hence, we have the proof. 
\end{proof}
Using Theorem \ref{Thm: Convergence_Main}, we can now compute the computation complexity (Definition \ref{Def: ComputationComplexity}) of the algorithm. 
\begin{cor}
\label{Cor: Computation_Complexity}
 For $\displaystyle \alpha = \frac{2}{3}$ and the rest of the parameters chosen according to Theorem \ref{Thm: Convergence_Main}.
 \begin{enumerate}[(i)]
    \item  For $K^{1 - \alpha} \geq 56b/2^{1/3}$, we have:
\begin{align*}
      \mathbb{E}\|  \nabla f(x_a) \|^2  & \leq  O \bigg(\frac{1 + \sigma^{4/3} + \sigma^2 \ln(T+1)}{K^{2/3} T}\bigg) + O \bigg(\frac{1 + \sigma^{4/3} + \sigma^2 \ln(T+1)}{K^{2/3} T^{2/3}}\bigg).
\end{align*}
and for $K^{1 - \alpha} \leq 56b/2^{1/3}$, we have
\begin{align*}
     \mathbb{E}\|  \nabla f({x}_{a}) \|^2  & \leq  O \bigg(\frac{1 + \sigma^{4/3} + \sigma^2 \ln(T+1)}{K T}\bigg) + O \bigg(\frac{1 + \sigma^{4/3} + \sigma^2 \ln(T+1)}{K^{2/3} T^{2/3}}\bigg).
\end{align*}
\item To reach an $\epsilon$-stationary solution (please see Definition \ref{Def: StationaryPt}) we need $\tilde{O}(K^{-1} \epsilon^{-3})$, gradient computations at each node, thereby, achieving linear speedup with the number of WNs $K$ in the network.
\end{enumerate}
\end{cor}
\begin{proof}
We know from the statement of Theorem \ref{Thm: Convergence_Main} that $\displaystyle w_t \leq  \sigma^2 \max\bigg\{2 ,~ b^3 K^{3 \alpha} - t,  ~\frac{(56b)^3}{K^{3 - 3\alpha}} \bigg\}$. This implies that, for $t = 0$, in the worst case we will have $w_0 = O(b^3 K^{3\alpha})$. Note that here the worst case refers to the worst case speedup achievable in terms of the number of WNs, $K$, present in the network. Moreover, after a finite number of iterations, specifically, $T \geq b^3 K^{3\alpha}$, we will have $\displaystyle w_T \leq \sigma^2 \max\bigg\{2 , ~\frac{(56b)^3}{K^{3 - 3\alpha}} \bigg\}$. Now we consider two regimes as: \vspace{0.1 in}\\
{\bf Regime 1:} When we have $K^{1 - \alpha} \geq 56b/2^{1/3}$. \vspace{0.05 in}\\
This means $\displaystyle 2 \geq {(56b)^3}/{K^{3 - 3\alpha}}$ which further implies that we have $w_T = O(2 \sigma^2)$. So under Regime 1, using $w_T = 2 \sigma^2$ and $w_0 = b^3 K^{3\alpha}$ along with $\displaystyle \bar{\kappa} = \frac{b K^\alpha \sigma^{2/3}}{L}$ and $\displaystyle c \leq \frac{56L^2}{K}$ as given in the statement of Theorem \ref{Thm: Convergence_Main} in the result of Theorem \ref{Thm: Convergence_Main} we get:
    \begin{align*}
    \frac{1}{T} \sum_{t=1}^T   \mathbb{E}\|  \nabla f(\bar{x}_{t}) \|^2  & \leq  O \bigg(\frac{1}{K^\alpha T} + \frac{\sigma^{4/3}}{K^{\alpha} T} + \frac{\sigma^{2}\ln(T+1)}{K^{2 - 2\alpha} T} \bigg) + O \bigg(\frac{1}{K^\alpha T^{2/3}} + \frac{\sigma^{4/3}}{K^\alpha T^{2/3}} + \frac{\sigma^2\ln(T+1)}{K^{2 - 2\alpha} T^{2/3}} \bigg).
\end{align*}
Choosing $\displaystyle \alpha = \frac{2}{3}$ we get:
\begin{align*}
     \frac{1}{T} \sum_{t=1}^T   \mathbb{E}\|  \nabla f(\bar{x}_{t}) \|^2  & \leq  O \bigg(\frac{1 + \sigma^{4/3} + \sigma^2 \ln(T+1)}{K^{2/3} T}\bigg) + O \bigg(\frac{1 + \sigma^{4/3} + \sigma^2 \ln(T+1)}{K^{2/3} T^{2/3}}\bigg).
\end{align*}
Therefore, we have $(i)$ under Regime 1. Moreover, to achieve $\epsilon$-stationary solution we need:
\begin{align*}
 O \bigg(\frac{1 + \sigma^{4/3} + \sigma^2 \ln(T+1)}{K^{2/3} T^{2/3}} \bigg) \leq \epsilon \qquad \Rightarrow \qquad T \geq \tilde{O}(K^{-1} \epsilon^{-3/2}).
\end{align*}
Now, we consider Regime 2 as:\vspace{0.1 in}\\
{\bf Regime 2:} When we have $K^{1 - \alpha} \leq 56b/2^{1/3}$. \vspace{0.05 in}\\
This means $\displaystyle 2 \leq {(56b)^3}/{K^{3 - 3\alpha}}$ which further implies that we have $w_T = O(\sigma^2/ K^{3-3\alpha})$.  

So under Regime 2, using $w_T =  \sigma^2/ K^{3-3\alpha}$ along with $\displaystyle \bar{\kappa} = \frac{b K^\alpha \sigma^{2/3}}{L}$ and $\displaystyle c \leq \frac{56L^2}{K}$ as given in the statement of Theorem \ref{Thm: Convergence_Main} in the result of Theorem \ref{Thm: Convergence_Main} we get:
    \begin{align*}
  \frac{1}{T} \sum_{t=1}^T   \mathbb{E}\|  \nabla f(\bar{x}_{t}) \|^2  & \leq  O \bigg(\frac{1}{K  T} + \frac{\sigma^{4/3}}{K T} + \frac{\sigma^{2}\ln(T+1)}{K^{3 - 3\alpha} T} \bigg) + O \bigg(\frac{1}{K^\alpha T^{2/3}} + \frac{\sigma^{4/3}}{K^\alpha T^{2/3}} + \frac{\sigma^2\ln(T+1)}{K^{2 - 2\alpha} T^{2/3}} \bigg).
\end{align*}
Again, choosing $\displaystyle \alpha = \frac{2}{3}$ we get:
\begin{align*}
  \mathbb{E}\|  \nabla f(\bar{x}_a) \|^2  =   \frac{1}{T} \sum_{t=1}^T   \mathbb{E}\|  \nabla f(\bar{x}_{t}) \|^2  & \leq  O \bigg(\frac{1 + \sigma^{4/3} + \sigma^2 \ln(T+1)}{K T}\bigg) + O \bigg(\frac{1 + \sigma^{4/3} + \sigma^2 \ln(T+1)}{K^{2/3} T^{2/3}}\bigg).
\end{align*}
Therefore, we have $(i)$ under Regime 2. Finally, using Jensen's inequality for the norm, i.e., $\big(\mathbb{E} \|\nabla f(x_a) \| \big)^2 \leq \mathbb{E} \|\nabla f(x_a) \|^2$, we get
\begin{align*}
        \mathbb{E}\|  \nabla f(\bar{x}_a) \|  & \leq  O \bigg(\frac{1 + \sigma^{2/3} + \sigma \sqrt{\ln(T+1)}}{K^{1/2} T^{1/2}}\bigg) + O \bigg(\frac{1 + \sigma^{2/3} + \sigma \sqrt{\ln(T+1)}}{K^{1/3} T^{1/3}}\bigg).
\end{align*}
Moreover, to achieve $\epsilon$-stationary solution we need
\begin{align*}
 O \bigg(\frac{1 + \sigma^{2/3} + \sigma \sqrt{ \ln(T+1)}}{K^{1/3} T^{1/3}} \bigg) \leq \epsilon \qquad \Rightarrow \qquad T \geq \tilde{O}(K^{-1} \epsilon^{-3}).
\end{align*}
Therefore, we have the corollary.
\end{proof}
\begin{rem}
Corollary \ref{Cor: Computation_Complexity} again implies that the total number of gradient evaluations at each WN in the worst case is reduced by a factor of $K$ for D-STORM when compared to the centralized version of the algorithm \cite{Cutkosky_NIPS2019}. This again implies that D-STORM is also capable of achieving linear speedup with the number of WNs, $K$, while at the same time achieving optimal computational complexity compared to the state-of-the-art up to logarithmic factors \cite{Arjevani_Carmon_2019_LowerBounds}. 
\end{rem}
\section{Conclusion}
\label{sec: Conclusion}
In this work, we proposed two distributed algorithms for stochastic non-convex optimization. The proposed algorithms AD-STORM and D-STORM are non-trivial extensions of the STORM algorithm proposed in \cite{Cutkosky_NIPS2019}. In contrast to the existing approaches, the proposed algorithms utilize momentum based construction of descent direction and execute in a ``single loop" which eliminates the need of computing large batch sizes to achieve variance reduction. Moreover, the ``adaptive" version of the algorithm utilizes the current gradient information across all WNs to design adaptive step-sizes. Importantly, we showed that the proposed algorithms achieve optimal computational complexity while attaining linear speedup with the number of WNs. Moreover, our approach did not assume identical data distributions across WNs making the approach general enough for federated learning applications. The future extensions of the proposed work include developing restarted versions of AD-STORM and D-STORM to improve the communication complexity of the proposed algorithms \cite{Yu_Zhu_2018parallel, Wang_Joshi_Arxiv_2018cooperative, Yu_Jin_Arxiv_2019linear}. Moreover, the extension of the proposed algorithms to decentralized (server less) architectures is also desirable.

\bibliographystyle{IEEEtran}
\bibliography{abrv,References}
\pagebreak 
\appendices
\section{AD-STORM}
\begin{lemma}
\label{Lem: AD_InnerProduct_e_t_Grad}
For $\bar{e}_{t-1} = \bar{d}_{t-1} - \nabla f(\bar{x}_{t-1})$ we have
\begin{align*}
\mathbb{E} \Bigg[ \frac{1}{\eta_{t-1}}\left\langle  (1 - a_t) \bar{e}_{t-1}, \frac{1}{K}\sum_{k = 1}^K \left[\left( \nabla f^{(k)}(\bar{x}_t;\xi_t^{(k)})  -  \nabla f^{(k)}(\bar{x}_t) \right)    - (1 - a_t) \left( \nabla f^{(k)}(\bar{x}_{t - 1}; \xi_t^{(k)}) - \nabla f^{(k)}(\bar{x}_{t - 1})\right) \right]   \right\rangle \Bigg] = 0
\end{align*}
\end{lemma}
\begin{proof}
Given $\bar{x}_t$ and the past, $\bar{e}_{t-1}$ which is given as  $\bar{d}_{t-1} - \nabla f(\bar{x}_{t-1})$ and $a_t = c^2 \eta_{t-1}$ is fixed as $\eta_{t-1}$ is fixed. Therefore, we can write
\begin{align*}
& \mathbb{E} \Bigg[ \frac{1}{\eta_{t-1}} \left\langle  (1 - a_t) \bar{e}_{t-1}, \frac{1}{K}\sum_{k = 1}^K \left[\left( \nabla f^{(k)}(\bar{x}_t;\xi_t^{(k)})  -  \nabla f^{(k)}(\bar{x}_t) \right)    - (1 - a_t) \left( \nabla f^{(k)}(\bar{x}_{t - 1}; \xi_t^{(k)}) - \nabla f^{(k)}(\bar{x}_{t - 1})\right) \right]   \right\rangle \Bigg] \\
& = \mathbb{E} \Bigg[ \frac{1}{\eta_{t-1}} \bigg\langle  (1 - a_t) \bar{e}_{t-1}, \frac{1}{K}\sum_{k = 1}^K \mathbb{E}\Big[\big( \nabla f^{(k)}(\bar{x}_t;\xi_t^{(k)}) -  \nabla f^{(k)}(\bar{x}_t) \big) \\
& \qquad \qquad \qquad \qquad \qquad \qquad \qquad \qquad \qquad    - (1 - a_t) \left( \nabla f^{(k)}(\bar{x}_{t - 1}; \xi_t^{(k)}) - \nabla f^{(k)}(\bar{x}_{t - 1})\right) \bigg| \bar{x_t}~\text{and past} \Big]  \bigg\rangle \Bigg].
\end{align*}
Note from Assumption \ref{Ass: Unbiased_Var_Grad}, given $\bar{x}_t$ and the past we have: $\mathbb{E} [ \nabla f(\bar{x}_t ; \xi_t^{(k)})] =  \nabla f(\bar{x}_t)$ for all $k \in [K]$. This implies that we have: 
$$\mathbb{E} \left[ \left( \nabla f^{(k)}(\bar{x}_t;\xi_t^{(k)})  -  \nabla f^{(k)}(\bar{x}_t) \right)    - (1 - a_t) \left( \nabla f^{(k)}(\bar{x}_{t - 1}; \xi_t^{(k)}) - \nabla f^{(k)}(\bar{x}_{t - 1})\right) \big| \bar{x}_t~\text{and past} \right] = 0$$
for all $k \in [K]$.

Therefore, we have the result.
\end{proof}

\begin{lemma}
\label{Lem: AD_InnerProd_AcrossNodes}
For $k, l \in [K]$ and $k \neq l$, 
\begin{enumerate}[(i)]
    \item We have,
    \begin{align*}
&  \mathbb{E} \bigg[ \frac{1}{\eta_{t-1}} \bigg\langle   \left( \nabla f^{(k)}(\bar{x}_t;\xi_t^{(k)})  -  \nabla f^{(k)}(\bar{x}_t) \right)    - (1 - a_t) \left( \nabla f^{(k)}(\bar{x}_{t - 1}; \xi_t^{(k)}) - \nabla f^{(k)}(\bar{x}_{t - 1})\right),\\
   & \qquad \qquad \qquad  \qquad     \left( \nabla f^{(l)}(\bar{x}_t;\xi_t^{(l)})  -  \nabla f^{(l)}(\bar{x}_t) \right)    - (1 - a_t) \left( \nabla f^{(l)}(\bar{x}_{t - 1}; \xi_t^{(l)}) - \nabla f^{(l)}(\bar{x}_{t - 1})\right) \bigg\rangle \bigg] =0
\end{align*}
    \item and 
    \begin{align*}
&  \mathbb{E} \Big[ 2 c^2 \eta_{t-1}^3 \Big\langle    \nabla f^{(k)}(\bar{x}_t;\xi_t^{(k)})  -  \nabla f^{(k)}(\bar{x}_t) , \nabla f^{(l)}(\bar{x}_t;\xi_t^{(l)})  -  \nabla f^{(l)}(\bar{x}_t) \Big\rangle \Big] =0
\end{align*}
\end{enumerate}
\end{lemma}
\begin{proof}
Notice that given $\bar{x}_t$ and the past, $\eta_{t-1}$ is fixed and the samples $\xi_t^{(k)}$ and $\xi_t^{(l)}$ at the $k$th and the $l$th nodes are chosen uniformly randomly, and independent of each other for all $k,l \in [K]$ and $k \neq l$. Therefore, we have;
\begin{align*}
& \mathbb{E} \bigg[ \frac{1}{\eta_{t-1}} \bigg\langle   \left( \nabla f^{(k)}(\bar{x}_t;\xi_t^{(k)})  -  \nabla f^{(k)}(\bar{x}_t) \right)    - (1 - a_t) \left( \nabla f^{(k)}(\bar{x}_{t - 1}; \xi_t^{(k)}) - \nabla f^{(k)}(\bar{x}_{t - 1})\right),\\
   & \qquad \qquad \qquad  \qquad  \quad    \qquad \quad \left( \nabla f^{(l)}(\bar{x}_t;\xi_t^{(l)})  -  \nabla f^{(l)}(\bar{x}_t) \right)    - (1 - a_t) \left( \nabla f^{(l)}(\bar{x}_{t - 1}; \xi_t^{(l)}) - \nabla f^{(l)}(\bar{x}_{t - 1})\right) \bigg\rangle \bigg]\\
& = \mathbb{E} \bigg[ \frac{1}{\eta_{t-1}} \bigg\langle \mathbb{E} \left[ \left( \nabla f^{(k)}(\bar{x}_t;\xi_t^{(k)})  -  \nabla f^{(k)}(\bar{x}_t) \right)    - (1 - a_t) \left( \nabla f^{(k)}(\bar{x}_{t - 1}; \xi_t^{(k)}) - \nabla f^{(k)}(\bar{x}_{t - 1})\right) \bigg| \bar{x}_t~\text{and past}\right],\\
   & \qquad  \qquad    \qquad  \mathbb{E} \left[ \left( \nabla f^{(l)}(\bar{x}_t;\xi_t^{(l)})  -  \nabla f^{(l)}(\bar{x}_t) \right)    - (1 - a_t) \left( \nabla f^{(l)}(\bar{x}_{t - 1}; \xi_t^{(l)}) - \nabla f^{(l)}(\bar{x}_{t - 1})\right) \bigg| \bar{x}_t~\text{and past}\right] \bigg\rangle \bigg]
\end{align*}
Note from Assumption \ref{Ass: Unbiased_Var_Grad}, given $\bar{x}_t$ and the past $a_t = c \eta_{t-1}^2$ is fixed as $\eta_{t-1}$ is fixed, we have: $\mathbb{E} [ \nabla f(\bar{x}_t ; \xi_t^{(k)}) ] =  \nabla f(\bar{x}_t)$ for all $k \in [K]$. This implies that we have: 
$$\mathbb{E} \left[ \left( \nabla f^{(k)}(\bar{x}_t;\xi_t^{(k)})  -  \nabla f^{(k)}(\bar{x}_t) \right)    - (1 - a_t) \left( \nabla f^{(k)}(\bar{x}_{t - 1}; \xi_t^{(k)}) - \nabla f^{(k)}(\bar{x}_{t - 1})\right) \Big| \bar{x}_t~\text{and past} \right] = 0,$$
for all $k \in [K]$.

Therefore, we have the proof of $(i)$.

The proof of $(ii)$ uses the same argument as $(i)$: Given $\bar{x}_t$ and the past, $\eta_{t-1}$ is fixed and the samples $\xi_t^{(k)}$ and $\xi_t^{(l)}$ at the $k$th and the $l$th node are chosen uniformly randomly and independent from each other for $k,l \in [K]$ and $k \neq l$. Using the fact that $\mathbb{E}\big[\nabla f^{(k)}(\bar{x}_t ; \xi_t^{(k)}) \big| \bar{x}_t~\text{and past}\big] = \nabla f^{(k)}(\bar{x}_t)$ we get $(ii)$.
\end{proof}

\begin{lemma} For any $k \in [K]$, 
\label{Lem: AD_Mean_Variance_Bound}
\begin{enumerate}[(i)]
    \item We have
    \begin{align*}
      &  \mathbb{E} \bigg[ \frac{2 (1 - a_t)^2}{\eta_{t-1}K^2}    \big\| \big( \nabla f^{(k)}(\bar{x}_{t};\xi_t^{(k)})  -  \nabla f^{(k)}(\bar{x}_{t - 1}; \xi_t^{(k)}) \big)    -  \big(\nabla f^{(k)}(\bar{x}_{t}) - \nabla f^{(k)}(\bar{x}_{t - 1})\big) \big\|^2 \bigg] \\
        & \qquad \qquad \qquad \qquad \qquad \qquad \qquad \qquad \qquad \leq\mathbb{E} \bigg[\frac{2 (1 - a_t)^2}{\eta_{t-1}K^2}   \big\|   \nabla f^{(k)}(\bar{x}_t;\xi_t^{(k)})  -  \nabla f^{(k)}(\bar{x}_{t - 1}; \xi_t^{(k)}) \big\|^2  \bigg],
    \end{align*}
    \item and 
    \begin{align*}
       \mathbb{E}\bigg[ \frac{2 c^2 \eta_{t-1}^3}{K^2}  \big\|  \nabla f^{(k)}(\bar{x}_{t}; \xi_t^{(k)}) - \nabla f^{(k)}(\bar{x}_{t})   \big\|^2 \bigg] \leq \mathbb{E} \bigg[\frac{ 2 c^2 \eta_{t-1}^3}{K^2}  \big\| \nabla f^{(k)}(\bar{x}_t; \xi_t^{(k)})\big\|^2 \bigg].
    \end{align*} 
\end{enumerate}
\end{lemma}
\begin{proof}
Considering individual terms in $(i)$, we have:
\begin{align*}
    & \mathbb{E} \bigg[ \frac{2 (1 - a_t)^2}{\eta_{t-1}K^2} \big\| \big( \nabla f^{(k)}(\bar{x}_{t};\xi_t^{(k)})  -  \nabla f^{(k)}(\bar{x}_{t - 1}; \xi_t^{(k)}) \big)    -  \big(\nabla f^{(k)}(\bar{x}_{t}) - \nabla f^{(k)}(\bar{x}_{t - 1})\big) \big\|^2 \bigg] \\
    & \qquad  = \mathbb{E} \bigg[ \frac{2 (1 - a_t)^2}{\eta_{t-1}K^2} \big\|\nabla f^{(k)}(\bar{x}_{t};\xi_t^{(k)})  -  \nabla f^{(k)}(\bar{x}_{t - 1}; \xi_t^{(k)}) \big\|^2 \bigg] + \mathbb{E} \bigg[ \frac{2 (1 - a_t)^2}{\eta_{t-1}K^2} \big\|  \nabla f^{(k)}(\bar{x}_{t}) - \nabla f^{(k)}(\bar{x}_{t - 1})  \big\|^2 \bigg] \\
    &\qquad \qquad \qquad \qquad - 2 \mathbb{E} \bigg[ \frac{2 (1 - a_t)^2}{\eta_{t-1}K^2} \big\langle  \nabla f^{(k)}(\bar{x}_{t};\xi_t^{(k)})  -  \nabla f^{(k)}(\bar{x}_{t - 1}; \xi_t^{(k)})  ,   \nabla f^{(k)}(\bar{x}_{t}) - \nabla f^{(k)}(\bar{x}_{t - 1})  \big\rangle \bigg] \\
    &  \qquad  \overset{(a)}{=} \mathbb{E} \bigg[\frac{2 (1 - a_t)^2}{\eta_{t-1}K^2} \big\|\nabla f^{(k)}(\bar{x}_{t};\xi_t^{(k)})  -  \nabla f^{(k)}(\bar{x}_{t - 1}; \xi_t^{(k)}) \big\|^2 \bigg] - \mathbb{E} \bigg[ \frac{2 (1 - a_t)^2}{\eta_{t-1}K^2} \big\|  \nabla f^{(k)}(\bar{x}_{t}) - \nabla f^{(k)}(\bar{x}_{t - 1})  \big\|^2 \bigg]\\
    & \qquad\leq \mathbb{E} \bigg[ \frac{2 (1 - a_t)^2}{\eta_{t-1}K^2} \big\|\nabla f^{(k)}(\bar{x}_{t};\xi_t^{(k)}) -  \nabla f^{(k)}(\bar{x}_{t - 1}; \xi_t^{(k)}) \big\|^2 \bigg].
\end{align*}
where $(a)$ follows from the fact that: Note from Assumption \ref{Ass: Unbiased_Var_Grad}, given $\bar{x}_t$ and the past, we have: $\mathbb{E} [ \nabla f(\bar{x}_t ; \xi_t^{(k)}) ] =  \nabla f(\bar{x}_t)$ for all $k \in [K]$, moreover, $a_t = c \eta_{t-1}^2$ and $\eta_{t-1}$ are fixed. This implies that we have: $\mathbb{E}\big[ \nabla f^{(k)}(\bar{x}_{t};\xi_t^{(k)})  -  \nabla f^{(k)}(\bar{x}_{t - 1}; \xi_t^{(k)}) \big| \bar{x}_t~\text{and past} \big] = \nabla f^{(k)}(\bar{x}_{t})  -  \nabla f^{(k)}(\bar{x}_{t - 1})$, using this we get
\begin{align*}
   & \mathbb{E} \bigg[ \frac{2 (1 - a_t)^2}{\eta_{t-1}K^2} \big\langle  \nabla f^{(k)}(\bar{x}_{t};\xi_t^{(k)})  -  \nabla f^{(k)}(\bar{x}_{t - 1}; \xi_t^{(k)})  ,   \nabla f^{(k)}(\bar{x}_{t}) - \nabla f^{(k)}(\bar{x}_{t - 1})  \big\rangle \bigg] \\
    & = \mathbb{E} \bigg[ \frac{2 (1 - a_t)^2}{\eta_{t-1}K^2} \big\langle \mathbb{E}\big[ \nabla f^{(k)}(\bar{x}_{t};\xi_t^{(k)})  -  \nabla f^{(k)}(\bar{x}_{t - 1}; \xi_t^{(k)}) \big| \bar{x}_t~\text{and past} \big] ,   \nabla f^{(k)}(\bar{x}_{t}) - \nabla f^{(k)}(\bar{x}_{t - 1})  \big\rangle \bigg]\\
    & = \mathbb{E} \bigg[ \frac{2 (1 - a_t)^2}{\eta_{t-1}K^2}   \| \nabla f^{(k)}(\bar{x}_{t}) - \nabla f^{(k)}(\bar{x}_{t - 1})  \|^2 \bigg]\
\end{align*}
Therefore, using this in above, we have the proof of $(i)$. The result of $(ii)$ follows from argument similar to $(i)$.
\end{proof}

\begin{lemma}
\label{Lem: AD_e_bar_bound}
For $\bar{e}_1$ chosen according to Algorithm \ref{Algo_AD-STORM}, we have:
\begin{align*}
    \mathbb{E}\| \bar{e}_1 \| \leq \frac{\sigma^2}{K}.
\end{align*}
\end{lemma}
\begin{proof}
Using the definition of $\bar{e}_1$ we have:
\begin{align*}
\mathbb{E} \|    \bar{e}_1 \|^2 & = \mathbb{E} \| \bar{d}_1 - \nabla f(\bar{x}_1) \|^2 \\
 & \overset{(a)}{=} \mathbb{E} \bigg\| \frac{1}{K} \sum_{k=1}^K \nabla f^{(k)}(x_1^{(k)}; \xi_1^{(k)}) - \nabla f(\bar{x}_1)  \bigg\|^2 \\
 & = \mathbb{E} \bigg\| \frac{1}{K} \sum_{k=1}^K \big( \nabla f^{(k)}(\bar{x}_1; \xi_1^{(k)}) - \nabla f^{(k)}(\bar{x}_1) \big) \bigg\|^2 \\
 & \leq  \frac{1}{K^2} \sum_{k=1}^K \mathbb{E} \big\| \nabla f^{(k)}(\bar{x}_1; \xi_1^{(k)}) - \nabla f^{(k)}(\bar{x}_1)  \big\|^2  \\
 & \qquad \qquad \qquad \qquad + \frac{1}{K^2} \sum_{k,l \in [K], k \neq l} \mathbb{E} \underbrace{\langle \nabla f^{(k)}(\bar{x}_1; \xi_1^{(k)}) - \nabla f^{(k)}(\bar{x}_1) ,\nabla f^{(l)}(\bar{x}_1; \xi_1^{(l)}) - \nabla f^{(l)}(\bar{x}_1)    \rangle}_{=0}\\
& \overset{(b)}{\leq} \frac{\sigma^2}{K}.
\end{align*}
where $(a)$ follows from the definition of $\bar{d}_1$ in Algorithm \ref{Algo_AD-STORM} (and Algorithm \ref{Algo_D-STORM}) and $(b)$ follows from Assumption \ref{Ass: Unbiased_Var_Grad} and the following:

From Assumption \ref{Ass: Unbiased_Var_Grad}, given $\bar{x}_1$ we have: $\mathbb{E} \big[  \nabla f^{(k)}(\bar{x}_1;\xi_1^{(k)})  -  \nabla f^{(k)}(\bar{x}_1) )\big] = 0$, for all $k \in [K]$. Moreover, as discussed in the proof of Lemma \ref{Lem: InnerProd_AcrossNodes} above, given $\bar{x}_1$ the samples $\xi_1^{(k)}$ and $\xi_1^{(l)}$ at the $k$th and the $l$th nodes are chosen uniformly randomly, and independent of each other for all $k,l \in [K]$ and $k \neq l$.
\begin{align*}
   & \mathbb{E} \Big[ \big\langle \nabla f^{(k)}(\bar{x}_1; \xi_1^{(k)}) - \nabla f^{(k)}(\bar{x}_1) ,\nabla f^{(l)}(\bar{x}_1; \xi_1^{(l)}) - \nabla f^{(l)}(\bar{x}_1)  \big\rangle \Big] \\
   & \qquad \qquad \qquad =  \mathbb{E} \bigg[\Big\langle \underbrace{\mathbb{E} \big[  \nabla f^{(k)}(\bar{x}_1; \xi_1^{(k)}) - \nabla f^{(k)}(\bar{x}_1) \Big| \bar{x}_1 \big]}_{=0} ,\underbrace{\mathbb{E} \big[ \nabla f^{(l)}(\bar{x}_1; \xi_1^{(l)}) - \nabla f^{(l)}(\bar{x}_1)  \Big| \bar{x}_1\big]}_{=0}\Big\rangle \bigg] \\
   &= 0. 
\end{align*}
Therefore, we have the proof.
\end{proof}

 \section{D-STORM}
\begin{lemma}
\label{Lem: InnerProduct_e_t_Grad}
For $\bar{e}_{t-1} = \bar{d}_{t-1} - \nabla f(\bar{x}_{t-1})$ we have
\begin{align*}
\mathbb{E} \left\langle  (1 - a_t) \bar{e}_{t-1}, \frac{1}{K}\sum_{k = 1}^K \left[\left( \nabla f^{(k)}(\bar{x}_t;\xi_t^{(k)})  -  \nabla f^{(k)}(\bar{x}_t) \right)    - (1 - a_t) \left( \nabla f^{(k)}(\bar{x}_{t - 1}; \xi_t^{(k)}) - \nabla f^{(k)}(\bar{x}_{t - 1})\right) \right]   \right\rangle = 0
\end{align*}
\end{lemma}
\begin{proof}
Given $\bar{x}_t$ and the past $\bar{e}_{t-1}$ which is given as  $\bar{d}_{t-1} - \nabla f(\bar{x}_{t-1})$ is fixed. Therefore, we can write
\begin{align*}
& \mathbb{E} \left\langle  (1 - a_t) \bar{e}_{t-1}, \frac{1}{K}\sum_{k = 1}^K \left[\left( \nabla f^{(k)}(\bar{x}_t;\xi_t^{(k)})  -  \nabla f^{(k)}(\bar{x}_t) \right)    - (1 - a_t) \left( \nabla f^{(k)}(\bar{x}_{t - 1}; \xi_t^{(k)}) - \nabla f^{(k)}(\bar{x}_{t - 1})\right) \right]   \right\rangle \\
& = \mathbb{E} \bigg\langle  (1 - a_t) \bar{e}_{t-1}, \frac{1}{K}\sum_{k = 1}^K \mathbb{E}\Big[\big( \nabla f^{(k)}(\bar{x}_t;\xi_t^{(k)}) -  \nabla f^{(k)}(\bar{x}_t) \big) \\
& \qquad \qquad \qquad \qquad \qquad \qquad \qquad \qquad \qquad    - (1 - a_t) \left( \nabla f^{(k)}(\bar{x}_{t - 1}; \xi_t^{(k)}) - \nabla f^{(k)}(\bar{x}_{t - 1})\right) \bigg| \bar{x_t}~\text{and past} \Big]  \bigg\rangle.
\end{align*}
Note from Assumption \ref{Ass: Unbiased_Var_Grad}, given $\bar{x}_t$ and the past we have: $\mathbb{E} [ \nabla f(\bar{x}_t ; \xi_t^{(k)})] =  \nabla f(\bar{x}_t)$ for all $k \in [K]$. This implies that we have: 
$$\mathbb{E} \left[ \left( \nabla f^{(k)}(\bar{x}_t;\xi_t^{(k)})  -  \nabla f^{(k)}(\bar{x}_t) \right)    - (1 - a_t) \left( \nabla f^{(k)}(\bar{x}_{t - 1}; \xi_t^{(k)}) - \nabla f^{(k)}(\bar{x}_{t - 1})\right) \big| \bar{x}_t~\text{and past} \right] = 0$$
for all $k \in [K]$.

Therefore, we have the result.
\end{proof}

\begin{lemma}
\label{Lem: InnerProd_AcrossNodes}
For $k, l \in [K]$ and $k \neq l$, we have
\begin{align*}
&  \mathbb{E} \bigg\langle   \left( \nabla f^{(k)}(\bar{x}_t;\xi_t^{(k)})  -  \nabla f^{(k)}(\bar{x}_t) \right)    - (1 - a_t) \left( \nabla f^{(k)}(\bar{x}_{t - 1}; \xi_t^{(k)}) - \nabla f^{(k)}(\bar{x}_{t - 1})\right),\\
   & \qquad \qquad \qquad  \qquad  \quad    \qquad \quad \left( \nabla f^{(l)}(\bar{x}_t;\xi_t^{(l)})  -  \nabla f^{(l)}(\bar{x}_t) \right)    - (1 - a_t) \left( \nabla f^{(l)}(\bar{x}_{t - 1}; \xi_t^{(l)}) - \nabla f^{(l)}(\bar{x}_{t - 1})\right) \bigg\rangle =0
\end{align*}
\end{lemma}
\begin{proof}
Notice that given $\bar{x}_t$ and the past, the samples $\xi_t^{(k)}$ and $\xi_t^{(l)}$ at the $k$th and the $l$th nodes are chosen uniformly randomly, and independent of each other for all $k,l \in [K]$ and $k \neq l$. Therefore, we have;
\begin{align*}
& \mathbb{E} \bigg\langle   \left( \nabla f^{(k)}(\bar{x}_t;\xi_t^{(k)})  -  \nabla f^{(k)}(\bar{x}_t) \right)    - (1 - a_t) \left( \nabla f^{(k)}(\bar{x}_{t - 1}; \xi_t^{(k)}) - \nabla f^{(k)}(\bar{x}_{t - 1})\right),\\
   & \qquad \qquad \qquad  \qquad  \quad    \qquad \quad \left( \nabla f^{(l)}(\bar{x}_t;\xi_t^{(l)})  -  \nabla f^{(l)}(\bar{x}_t) \right)    - (1 - a_t) \left( \nabla f^{(l)}(\bar{x}_{t - 1}; \xi_t^{(l)}) - \nabla f^{(l)}(\bar{x}_{t - 1})\right) \bigg\rangle\\
& = \mathbb{E} \bigg\langle \mathbb{E} \left[ \left( \nabla f^{(k)}(\bar{x}_t;\xi_t^{(k)})  -  \nabla f^{(k)}(\bar{x}_t) \right)    - (1 - a_t) \left( \nabla f^{(k)}(\bar{x}_{t - 1}; \xi_t^{(k)}) - \nabla f^{(k)}(\bar{x}_{t - 1})\right) \bigg| \bar{x}_t~\text{and past}\right],\\
   & \qquad  \qquad    \qquad  \mathbb{E} \left[ \left( \nabla f^{(l)}(\bar{x}_t;\xi_t^{(l)})  -  \nabla f^{(l)}(\bar{x}_t) \right)    - (1 - a_t) \left( \nabla f^{(l)}(\bar{x}_{t - 1}; \xi_t^{(l)}) - \nabla f^{(l)}(\bar{x}_{t - 1})\right) \bigg| \bar{x}_t~\text{and past}\right] \bigg\rangle 
\end{align*}
Note from Assumption \ref{Ass: Unbiased_Var_Grad}, given $\bar{x}_t$ and the past we have: $\mathbb{E} [ \nabla f(\bar{x}_t ; \xi_t^{(k)}) ] =  \nabla f(\bar{x}_t)$ for all $k \in [K]$. This implies that we have: 
$$\mathbb{E} \left[ \left( \nabla f^{(k)}(\bar{x}_t;\xi_t^{(k)})  -  \nabla f^{(k)}(\bar{x}_t) \right)    - (1 - a_t) \left( \nabla f^{(k)}(\bar{x}_{t - 1}; \xi_t^{(k)}) - \nabla f^{(k)}(\bar{x}_{t - 1})\right) \Big| \bar{x}_t~\text{and past} \right] = 0,$$
for all $k \in [K]$.

Therefore, we have the proof.
\end{proof}

\begin{lemma} For any $k \in [K]$, we have
\label{Lem: Mean_Variance_Bound}
\begin{align*}
    & \mathbb{E} \big\| \big( \nabla f^{(k)}(\bar{x}_{t};\xi_t^{(k)})  -  \nabla f^{(k)}(\bar{x}_{t - 1}; \xi_t^{(k)}) \big)    -  \big(\nabla f^{(k)}(\bar{x}_{t}) - \nabla f^{(k)}(\bar{x}_{t - 1})\big) \big\|^2 \\
    & \qquad  \qquad \qquad \qquad \qquad \qquad \qquad \qquad \qquad \qquad \qquad \leq  \mathbb{E} \big\|   \nabla f^{(k)}(\bar{x}_{t};\xi_t^{(k)})  -  \nabla f^{(k)}(\bar{x}_{t - 1}; \xi_t^{(k)})   \big\|^2. 
\end{align*}
\end{lemma}
\begin{proof}
We have:
\begin{align*}
    & \mathbb{E} \big\| \big( \nabla f^{(k)}(\bar{x}_{t};\xi_t^{(k)})  -  \nabla f^{(k)}(\bar{x}_{t - 1}; \xi_t^{(k)}) \big)    -  \big(\nabla f^{(k)}(\bar{x}_{t}) - \nabla f^{(k)}(\bar{x}_{t - 1})\big) \big\|^2 \\
    & \qquad  = \mathbb{E} \big\|\nabla f^{(k)}(\bar{x}_{t};\xi_t^{(k)})  -  \nabla f^{(k)}(\bar{x}_{t - 1}; \xi_t^{(k)}) \big\|^2 + \mathbb{E} \big\|  \nabla f^{(k)}(\bar{x}_{t}) - \nabla f^{(k)}(\bar{x}_{t - 1})  \big\|^2 \\
    &\qquad \qquad \qquad \qquad - 2 \mathbb{E} \big\langle  \nabla f^{(k)}(\bar{x}_{t};\xi_t^{(k)})  -  \nabla f^{(k)}(\bar{x}_{t - 1}; \xi_t^{(k)})  ,   \nabla f^{(k)}(\bar{x}_{t}) - \nabla f^{(k)}(\bar{x}_{t - 1})  \big\rangle \\
    &  \qquad  \overset{(a)}{=} \mathbb{E} \big\|\nabla f^{(k)}(\bar{x}_{t};\xi_t^{(k)})  -  \nabla f^{(k)}(\bar{x}_{t - 1}; \xi_t^{(k)}) \big\|^2 - \mathbb{E} \big\|  \nabla f^{(k)}(\bar{x}_{t}) - \nabla f^{(k)}(\bar{x}_{t - 1})  \big\|^2\\
    & \qquad\leq \mathbb{E} \big\|\nabla f^{(k)}(\bar{x}_{t};\xi_t^{(k)}) -  \nabla f^{(k)}(\bar{x}_{t - 1}; \xi_t^{(k)}) \big\|^2.
\end{align*}
where $(a)$ follows from the discussion similar to in Lemma \ref{Lem: InnerProd_AcrossNodes} above as: Note from Assumption \ref{Ass: Unbiased_Var_Grad}, given $\bar{x}_t$ and the past we have: $\mathbb{E} [ \nabla f(\bar{x}_t ; \xi_t^{(k)}) ] =  \nabla f(\bar{x}_t)$ for all $k \in [K]$. This implies that we have: 
$$\mathbb{E} \big[ \nabla f^{(k)}(\bar{x}_{t};\xi_t^{(k)})  -  \nabla f^{(k)}(\bar{x}_{t - 1}; \xi_t^{(k)}) \big]  = \mathbb{E} \big[ \nabla f^{(k)}(\bar{x}_{t} )  -  \nabla f^{(k)}(\bar{x}_{t - 1} ) \big].$$
Therefore, we can write:
\begin{align*}
    & \mathbb{E} \big\langle  \nabla f^{(k)}(\bar{x}_{t};\xi_t^{(k)})  -  \nabla f^{(k)}(\bar{x}_{t - 1}; \xi_t^{(k)})  ,   \nabla f^{(k)}(\bar{x}_{t}) - \nabla f^{(k)}(\bar{x}_{t - 1})  \big\rangle \\
    & \qquad \qquad \qquad \qquad=  \mathbb{E} \big\langle   \mathbb{E} \big[ \nabla f^{(k)}(\bar{x}_{t};\xi_t^{(k)})  -  \nabla f^{(k)}(\bar{x}_{t - 1} ; \xi_t^{(k)}) \big| \bar{x}_t~\text{and past} \big] ,   \nabla f^{(k)}(\bar{x}_{t}) - \nabla f^{(k)}(\bar{x}_{t - 1})  \big\rangle \\
    & \qquad \qquad \qquad \qquad = \mathbb{E} \|  \nabla f^{(k)}(\bar{x}_{t}) - \nabla f^{(k)}(\bar{x}_{t - 1})  \|^2.
\end{align*}
Hence, we have the proof.
\end{proof}

\section{}

\begin{lemma}[From \cite{Cutkosky_NIPS2019}]
\label{Lem: AD_Sum_1overT}
Let $a_0 > 0$ and $a_1,a_2, \ldots, a_T \geq 0$. We have
$$\sum_{t=1}^T \frac{a_t}{a_0 + \sum_{i=t}^t a_i} \leq \ln \bigg(1 + \frac{\sum_{i=1}^t a_i}{a_0} \bigg).$$
\end{lemma}


	\begin{lemma}
		\label{Lem: Norm_Ineq}
		For $X_1, X_2, \ldots, X_n \in \mathbb{R}^d$, we have 
		\begin{align*}
			\|X_1 + X_2 + \ldots + X_n \|^2 \leq n \| X_1\|^2 + n \| X_2\|^2+ \ldots + n \| X_n\|^2.
		\end{align*}
	\end{lemma}

\end{document}